\documentclass[a4,12pt]{amsart}

\usepackage{amssymb}
\usepackage{amstext}
\usepackage{amsmath}
\usepackage{amscd}
\usepackage{amsthm}
\usepackage{amsfonts}
\usepackage{enumerate}
\usepackage{graphicx}
\usepackage{latexsym}
\usepackage[all]{xy}
\usepackage[usenames]{color}

\newtheorem*{corollary*}{Corollary}
\newtheorem{theorem}{Theorem}[section]

\newtheorem{lemma}[theorem]{Lemma}
\newtheorem{proposition}[theorem]{Proposition}
\newtheorem*{proposition*}{Proposition}

\newtheorem{conjecture}[theorem]{Conjecture}

\newtheorem*{claim*}{Claim}

\theoremstyle{definition}
\newtheorem{definition}[theorem]{Definition}
\newtheorem{remark}[theorem]{Remark}
\newtheorem{example}[theorem]{Example}

\newtheorem{definitiontheorem}[theorem]{Definition-Theorem}

\theoremstyle{remark}

\numberwithin{equation}{theorem}
\renewcommand{\mod}{\operatorname{mod}}

\newcommand{\End}{\operatorname{End}}
\newcommand{\Hom}{\operatorname{Hom}}
\newcommand{\add}{\operatorname{\mathsf{add}}}
\newcommand{\Ext}{\operatorname{Ext}}

\newcommand{\Tau}{\mathsf{T}}

\newcommand{\tilt}{\operatorname{\mathsf{tilt}}}

\newcommand{\tsilt}{\operatorname{\mathsf{2silt}}}
\newcommand{\stilt}{\operatorname{\mathsf{s\text{-}tilt}}}
\newcommand{\sttilt}{\operatorname{\mathsf{s\tau -tilt}}}

\newcommand{\Z}{{\mathbb{Z}}}
\newcommand{\N}{{\mathbb{N}}}

\newcommand{\Ker}{\operatorname{Ker}}

\newcommand{\ind}{\operatorname{ind}}
\newcommand{\dimvec}{\operatorname{\underline{dim}}}

\begin{document}
\title{Remarks on lengths of maximal green sequences for quivers of type $\tilde{A}_{n,1}$  }
\author{Ryoichi Kase}
\address{Department of Mathematics,
Nara Women's University, Kitauoya-Nishimachi, Nara city, Nara 630-8506, Japan}
\email{r-kase@cc.nara-wu.ac.jp}
\begin{abstract}
A maximal green sequence introduced by B. Keller is a certain sequence of quiver mutations at green vertices.  T.~Br$\ddot{\mathrm{u}}$stle, G.~Dupont and M.~P$\acute{\mathrm{e}}$rotin
showed that for an acyclic quiver, maximal green sequences are realized as maximal paths in the Hasse quiver of the poset of support tilting modules. In \cite{BDP}, they considered
possible lengths of maximal green sequences. In this paper, we calculate possible lengths of maximal green sequences for a quiver of type $A$ or of type $\tilde{A}_{n,1}$ by using theory of tilting mutation.        
\end{abstract}
\maketitle
\section*{Introduction}
A green quiver mutation was introduced by B. Keller \cite{Ke}. It was special class of quiver mutations.
Then maximal green sequences are defined to be maximal sequences of green quiver mutations.   
Let $Q$ be a cluster quiver that is a finite connected quiver without loops or oriented cycles.
 In this setting, green quiver mutations induce a quiver\;(or an oriented graph)
$\overrightarrow{E}(Q)$ which has the cluster exchange graph corresponding to $Q$ introduced in \cite{FZ} as the underlying graph.
 Moreover if $Q$ is acyclic, then a green mutation for $Q$ closely related to theory of tilting mutation. 

The notion of tilting modules for finite dimensional algebras were introduced by B. Brenner and M. C. R. Butler \cite{BB}. 
Since tilting modules control derived equivalence, to obtain many tilting modules is an important problem for representation theory of finite dimensional algebras. 
 Tilting mutation introduced by C. Riedtmann and A. Schofield is an approach to this problem \cite{RS}. It is an operation which gives a new tilting modules from given one by
replacing an indecomposable direct summand. They also introduced a tilting quiver corresponding to tilting mutations. D. Happel and L. Unger defined partial order on the set of basic tilting modules and showed that the tilting quiver coincides
with the Hasse quiver of this poset. 
 The poset of support tilting modules of a finite dimensional path algebra $kQ$ was defined by C. Ingalls and H. Thomas as an extension of the poset of tilting modules \cite{IT}. Then its Hasse quiver is isomorphic to an oriented exchange graph $\overrightarrow{E}(Q)$\;\cite{AIR,BDP}.

For lengths of maximal green sequences, T.~Br$\ddot{\mathrm{u}}$stle, G.~Dupont and M.~P$\acute{\mathrm{e}}$rotin conjectured the following.
\begin{conjecture}\cite{BDP}
If $Q$ is a cluster quiver, then possible lengths of maximal green sequences for $Q$ form an interval in $\Z$.
\end{conjecture}

In this article, we check that above conjecture holds true if $Q$ is either a quiver of type $A_{n}$ or a quiver of type $\tilde{A}_{n,1}$. We also
give  possible lengths of maximal green sequences for a quiver of type $\tilde{A}_{n,1}$.

\subsection*{Notation} Throughout this paper, let $\Lambda$ be a finite dimensional algebra over an algebraically
closed field $k$.
\begin{itemize}
\item[1.] We denote by $\mod \Lambda$ the category of finite dimensional right $\Lambda$-modules.
\item[2.] We denote by $\tau$ the Auslander-Reiten translation of $\mod \Lambda$ (refer to \cite{ASS, ARS} for definition and properties).
\item[3.] For two integers $a\leq b$, we denote by $[a,b]$ the interval $\{a,a+1,\dots, b\}$. 
\end{itemize} 
 
\section{Preliminary}
\subsection{Green mutations and maximal green sequences}
In this subsection, we recall fundamentals of green quiver mutations and maximal green sequences.
Definitions and results in this subsection are referred to \cite{BDP}. 
\begin{definition}Let $Q$ be a finite connected quiver.
\begin{itemize}
\item[(1)] $Q$ is a cluster quiver if $Q$ does not admit loops or 2-oriented cycles.
\item[(2)] An ice quiver is a pair $(Q,F)$ where $Q$ is a cluster quiver and $F$ is a subset of $Q_{0}$
such that there is no arrow between two vertices of $F$. For an ice quiver $(Q,F)$, we call a vertex in $F$
a frozen vertex. 
\end{itemize}
\end{definition}
\begin{definition}
Let $(Q,F)$ be an ice quiver and $k\in Q_{0}\setminus F$. We define a new ice quiver $(\mu_{k}Q,F)$ from $(Q,F)$ by applying
following $4$-steps.
\begin{itemize}
\item[$(\mathrm{Step}\;1)$] For any pair of arrows $i\stackrel{\alpha}{\to} k\stackrel{\beta}{\to} j$, add an arrow $i\stackrel{[\alpha\beta]}{\to}j$.
\item[$(\mathrm{Step}\;2)$] Replace any arrow $i\stackrel{\alpha}{\to} k$ by an arrow $i\stackrel{\alpha^{\ast}}{\leftarrow} k$. 
\item[$(\mathrm{Step}\;3)$] Replace any arrow $k\stackrel{\beta}{\to} j$ by an arrow $k\stackrel{\beta^{\ast}}{\leftarrow} j$.
\item[$(\mathrm{Step}\;4)$] Remove a maximal collection of $2\text{-}cycles$ and all arrows between frozen vertices.
\end{itemize}
We call $(\mu_{k}Q,F)$ the mutation of $(Q,F)$ at a non-frozen vertex $k$.
\end{definition}
\begin{definition}Let $(Q,F),\;(Q',F)$ be ice quivers with $Q_{0}=Q'_{0}$.
\begin{itemize}
\item[(1)] $(Q,F)$ and $(Q',F)$ are mutation-equivalent if there is a finite sequence
$(k_{1},\dots,k_{l})$ of non-frozen vetices such that 
\[(Q',F)=(\mu_{k_{l}}\mu_{k_{l-1}}\cdots\mu_{k_{1}}Q,F).\] 
Then we denote by $\mathrm{Mut}(Q,F)$ the mutation-equivalence class of $(Q,F)$.
\item[(2)]$(Q,F)$ is isomorphic to $(Q',F)$ as ice quivers if there is an isomorphism
$\varphi:Q\rightarrow Q'$ of quivers fixing any frozen vertex. In this case, we denote
$(Q,F)\simeq (Q',F)$ and denote by $[(Q,F)]$ the isomorphism class of $(Q,F)$.
\end{itemize}

\end{definition}
From now on, we assume that $Q$ is a cluster quiver and $Q'_{0}:=\{c(i)\mid i\in Q_{0}\}$ is a copy of $Q_{0}$.
\begin{definition}
The framed quiver associated with $Q$ is the quiver $\hat{Q}$ defined as follows:
\begin{itemize}
\item $\hat{Q}_{0}:=Q_{0}\sqcup Q'_{0}$.
\item $\hat{Q}_{1}:=Q_{1}\sqcup \{i\to c(i)\mid i\in Q_{0}\}$.
\end{itemize}
The coframed quiver associated with $Q$ is the quiver $\check{Q}$ defined as follows:
\begin{itemize}
\item $\check{Q}_{0}:=Q_{0}\sqcup Q'_{0}$.
\item $\check{Q}_{1}:=Q_{1}\sqcup \{c(i)\to i\mid i\in Q_{0}\}$.
\end{itemize}
Note that $(\hat{Q},Q'_{0})$ and $(\check{Q},Q'_{0})$ are ice quivers. We denote by $\mathrm{Mut}(\hat{Q})$ the
mutation-equivalence class $\mathrm{Mut}(\hat{Q},Q'_{0})$ of $(\hat{Q},Q'_{0})$.  
\end{definition}
\begin{definition}
Let $R\in \mathrm{Mut}(\hat{Q})$ and let $i$ be a non-frozen vertex.
\begin{itemize}
\item[(1)] $i$ is said to be a green vertex if $\{\alpha\in R_{1}\mid s(\alpha)\in Q'_{0},\;t(\alpha)=i\}=\emptyset$.
\item[(2)] $i$ is said to be a red vertex if $\{\alpha\in R_{1}\mid  s(\alpha)=i,\;t(\alpha)\in Q'_{0}\}=\emptyset$.
\end{itemize}  
\end{definition}
\begin{theorem}
\cite{BDP}
Let $R\in \mathrm{Mut}(\hat{Q})$. Then we have \[R_{0}=\{i\in R_{0}\mid i\ \mathrm{is\ green}\}\sqcup\{i\in R_{0}\mid i\ \mathrm{is\ red}\}.\]
\end{theorem}
\begin{definition}
A green sequence for a cluster quiver $Q$ is a sequence $\mathbf{i}=(i_{1},i_{2},\dots,i_{l})$ of $Q_{0}$
such that $i_{1}$ is green in $\hat{Q}$ and for any $2\leq k \leq l$, $i_{k}$ is green in $\mu_{i_{k-1}}\cdots\mu_{1}\hat{Q}$.
In this case, $l$ is called length of $\mathbf{i}$ and we denote it by $l(\mathbf{i})$. A green sequence $\mathbf{i}=(i_{1},i_{2},\dots,i_{l})$ is said to be maximal if
 $\mu_{i_{l}}\cdots\mu_{i_{1}}\hat{Q}$ has no green vertex. Then we denote by $\mathbf{green}(Q)$ the set of maximal green sequences for $Q$ and $\mathbf{green }_{l}(Q):=\{\mathbf{i}\in \mathbf{green}(Q) \mid l(\mathbf{i})=l\}.$
\end{definition}
\begin{conjecture}
\label{interval conj}\cite[Conjecure\;2.22]{BDP}
Let $Q$ be a cluster quiver. Then $\{l\in \Z_{\geq 0}\mid \mathbf{green}_{l}(Q)\neq \emptyset \}$ is an interval in $\Z$.
\end{conjecture}
\begin{definition}
An oriented exchange graph $\overrightarrow{ \mathrm{E}}(Q)$ is defined as follows:
\begin{itemize}
\item $\overrightarrow{ \mathrm{E}}(Q)_{0}:= \mathrm{Mut}(\hat{Q})/\simeq$.
\item $[R]\to [R']$ in $\overrightarrow{ \mathrm{E}}(Q)$ if $[R']=[\mu_{k}R]$ for some green vertex $k$ of $R$.
\end{itemize} 
\end{definition}
\begin{example}
\label{example1} Let $Q=1\to 2$ and $Q'_{0}=\{c(1)=3,c(2)=4\}$. Then an orientation exchange graph $\overrightarrow{\mathrm{E}}(Q)$ is given by the following.
\[ 
{\unitlength 0.1in%
\begin{picture}( 48.1000, 37.0000)(  8.8000,-40.9000)%
\put(28.0000,-10.0000){\makebox(0,0)[lb]{$1$}}%
\put(34.0000,-10.0000){\makebox(0,0)[lb]{$2$}}%
\put(34.0000,-6.0000){\makebox(0,0)[lb]{$4$}}%
\put(28.0000,-6.0000){\makebox(0,0)[lb]{$3$}}%
\put(10.0000,-24.0000){\makebox(0,0)[lb]{$1$}}%
\put(16.0000,-24.0000){\makebox(0,0)[lb]{$2$}}%
\put(16.0000,-20.0000){\makebox(0,0)[lb]{$4$}}%
\put(10.0000,-20.0000){\makebox(0,0)[lb]{$3$}}%
\put(48.0000,-18.0000){\makebox(0,0)[lb]{$1$}}%
\put(54.0000,-18.0000){\makebox(0,0)[lb]{$2$}}%
\put(54.0000,-14.0000){\makebox(0,0)[lb]{$4$}}%
\put(48.0000,-14.0000){\makebox(0,0)[lb]{$3$}}%
\put(48.0000,-30.0000){\makebox(0,0)[lb]{$1$}}%
\put(54.0000,-30.0000){\makebox(0,0)[lb]{$2$}}%
\put(54.0000,-26.0000){\makebox(0,0)[lb]{$4$}}%
\put(48.0000,-26.0000){\makebox(0,0)[lb]{$3$}}%
\put(20.0000,-40.0000){\makebox(0,0)[lb]{$1$}}%
\put(26.0000,-40.0000){\makebox(0,0)[lb]{$2$}}%
\put(26.0000,-36.0000){\makebox(0,0)[lb]{$4$}}%
\put(20.0000,-36.0000){\makebox(0,0)[lb]{$3$}}%
\put(36.0000,-40.0000){\makebox(0,0)[lb]{$1$}}%
\put(42.0000,-40.0000){\makebox(0,0)[lb]{$2$}}%
\put(42.0000,-36.0000){\makebox(0,0)[lb]{$4$}}%
\put(36.0000,-36.0000){\makebox(0,0)[lb]{$3$}}%
%
\special{pn 8}%
\special{pa 2680 390}%
\special{pa 3690 390}%
\special{pa 3690 1090}%
\special{pa 2680 1090}%
\special{pa 2680 390}%
\special{pa 3690 390}%
\special{fp}%
%
\special{pn 8}%
\special{pa 880 1790}%
\special{pa 1890 1790}%
\special{pa 1890 2490}%
\special{pa 880 2490}%
\special{pa 880 1790}%
\special{pa 1890 1790}%
\special{fp}%
%
\special{pn 8}%
\special{pa 1880 3390}%
\special{pa 2890 3390}%
\special{pa 2890 4090}%
\special{pa 1880 4090}%
\special{pa 1880 3390}%
\special{pa 2890 3390}%
\special{fp}%
%
\special{pn 8}%
\special{pa 3480 3390}%
\special{pa 4490 3390}%
\special{pa 4490 4090}%
\special{pa 3480 4090}%
\special{pa 3480 3390}%
\special{pa 4490 3390}%
\special{fp}%
%
\special{pn 8}%
\special{pa 4680 1190}%
\special{pa 5690 1190}%
\special{pa 5690 1890}%
\special{pa 4680 1890}%
\special{pa 4680 1190}%
\special{pa 5690 1190}%
\special{fp}%
%
\special{pn 8}%
\special{pa 4680 2390}%
\special{pa 5690 2390}%
\special{pa 5690 3090}%
\special{pa 4680 3090}%
\special{pa 4680 2390}%
\special{pa 5690 2390}%
\special{fp}%
%
\special{pn 8}%
\special{pa 3020 3800}%
\special{pa 3420 3800}%
\special{fp}%
%
\special{pn 8}%
\special{pa 3040 3730}%
\special{pa 3069 3710}%
\special{pa 3097 3695}%
\special{pa 3126 3690}%
\special{pa 3155 3697}%
\special{pa 3184 3712}%
\special{pa 3214 3730}%
\special{pa 3243 3748}%
\special{pa 3273 3759}%
\special{pa 3302 3760}%
\special{pa 3332 3751}%
\special{pa 3361 3736}%
\special{pa 3400 3710}%
\special{fp}%
%
\special{pn 8}%
\special{pa 2940 940}%
\special{pa 3340 940}%
\special{fp}%
\special{sh 1}%
\special{pa 3340 940}%
\special{pa 3273 920}%
\special{pa 3287 940}%
\special{pa 3273 960}%
\special{pa 3340 940}%
\special{fp}%
%
\special{pn 8}%
\special{pa 2850 880}%
\special{pa 2850 620}%
\special{fp}%
\special{sh 1}%
\special{pa 2850 620}%
\special{pa 2830 687}%
\special{pa 2850 673}%
\special{pa 2870 687}%
\special{pa 2850 620}%
\special{fp}%
%
\special{pn 8}%
\special{pa 3450 880}%
\special{pa 3450 620}%
\special{fp}%
\special{sh 1}%
\special{pa 3450 620}%
\special{pa 3430 687}%
\special{pa 3450 673}%
\special{pa 3470 687}%
\special{pa 3450 620}%
\special{fp}%
%
\special{pn 8}%
\special{pa 1630 2260}%
\special{pa 1630 2000}%
\special{fp}%
\special{sh 1}%
\special{pa 1630 2000}%
\special{pa 1610 2067}%
\special{pa 1630 2053}%
\special{pa 1650 2067}%
\special{pa 1630 2000}%
\special{fp}%
%
\special{pn 8}%
\special{pa 1030 2020}%
\special{pa 1030 2280}%
\special{fp}%
\special{sh 1}%
\special{pa 1030 2280}%
\special{pa 1050 2213}%
\special{pa 1030 2227}%
\special{pa 1010 2213}%
\special{pa 1030 2280}%
\special{fp}%
%
\special{pn 8}%
\special{pa 1510 2340}%
\special{pa 1110 2340}%
\special{fp}%
\special{sh 1}%
\special{pa 1110 2340}%
\special{pa 1177 2360}%
\special{pa 1163 2340}%
\special{pa 1177 2320}%
\special{pa 1110 2340}%
\special{fp}%
%
\special{pn 8}%
\special{pa 2130 3950}%
\special{pa 2530 3950}%
\special{fp}%
\special{sh 1}%
\special{pa 2530 3950}%
\special{pa 2463 3930}%
\special{pa 2477 3950}%
\special{pa 2463 3970}%
\special{pa 2530 3950}%
\special{fp}%
%
\special{pn 8}%
\special{pa 2050 3630}%
\special{pa 2050 3890}%
\special{fp}%
\special{sh 1}%
\special{pa 2050 3890}%
\special{pa 2070 3823}%
\special{pa 2050 3837}%
\special{pa 2030 3823}%
\special{pa 2050 3890}%
\special{fp}%
%
\special{pn 8}%
\special{pa 2640 3620}%
\special{pa 2640 3880}%
\special{fp}%
\special{sh 1}%
\special{pa 2640 3880}%
\special{pa 2660 3813}%
\special{pa 2640 3827}%
\special{pa 2620 3813}%
\special{pa 2640 3880}%
\special{fp}%
%
\special{pn 8}%
\special{pa 5350 1740}%
\special{pa 4950 1740}%
\special{fp}%
\special{sh 1}%
\special{pa 4950 1740}%
\special{pa 5017 1760}%
\special{pa 5003 1740}%
\special{pa 5017 1720}%
\special{pa 4950 1740}%
\special{fp}%
%
\special{pn 8}%
\special{pa 4840 1680}%
\special{pa 4840 1420}%
\special{fp}%
\special{sh 1}%
\special{pa 4840 1420}%
\special{pa 4820 1487}%
\special{pa 4840 1473}%
\special{pa 4860 1487}%
\special{pa 4840 1420}%
\special{fp}%
%
\special{pn 8}%
\special{pa 5440 1410}%
\special{pa 5440 1670}%
\special{fp}%
\special{sh 1}%
\special{pa 5440 1670}%
\special{pa 5460 1603}%
\special{pa 5440 1617}%
\special{pa 5420 1603}%
\special{pa 5440 1670}%
\special{fp}%
%
\special{pn 8}%
\special{pa 4930 1670}%
\special{pa 5380 1380}%
\special{fp}%
\special{sh 1}%
\special{pa 5380 1380}%
\special{pa 5313 1399}%
\special{pa 5335 1409}%
\special{pa 5335 1433}%
\special{pa 5380 1380}%
\special{fp}%
%
\special{pn 8}%
\special{pa 4860 2610}%
\special{pa 4860 2870}%
\special{fp}%
\special{sh 1}%
\special{pa 4860 2870}%
\special{pa 4880 2803}%
\special{pa 4860 2817}%
\special{pa 4840 2803}%
\special{pa 4860 2870}%
\special{fp}%
%
\special{pn 8}%
\special{pa 5360 2560}%
\special{pa 4910 2850}%
\special{fp}%
\special{sh 1}%
\special{pa 4910 2850}%
\special{pa 4977 2831}%
\special{pa 4955 2821}%
\special{pa 4955 2797}%
\special{pa 4910 2850}%
\special{fp}%
%
\special{pn 8}%
\special{pa 4940 2940}%
\special{pa 5340 2940}%
\special{fp}%
\special{sh 1}%
\special{pa 5340 2940}%
\special{pa 5273 2920}%
\special{pa 5287 2940}%
\special{pa 5273 2960}%
\special{pa 5340 2940}%
\special{fp}%
%
\special{pn 8}%
\special{pa 5370 2890}%
\special{pa 4920 2600}%
\special{fp}%
\special{sh 1}%
\special{pa 4920 2600}%
\special{pa 4965 2653}%
\special{pa 4965 2629}%
\special{pa 4987 2619}%
\special{pa 4920 2600}%
\special{fp}%
%
\special{pn 8}%
\special{pa 4150 3940}%
\special{pa 3750 3940}%
\special{fp}%
\special{sh 1}%
\special{pa 3750 3940}%
\special{pa 3817 3960}%
\special{pa 3803 3940}%
\special{pa 3817 3920}%
\special{pa 3750 3940}%
\special{fp}%
%
\special{pn 8}%
\special{pa 3710 3600}%
\special{pa 4160 3890}%
\special{fp}%
\special{sh 1}%
\special{pa 4160 3890}%
\special{pa 4115 3837}%
\special{pa 4115 3861}%
\special{pa 4093 3871}%
\special{pa 4160 3890}%
\special{fp}%
%
\special{pn 8}%
\special{pa 4150 3610}%
\special{pa 3700 3900}%
\special{fp}%
\special{sh 1}%
\special{pa 3700 3900}%
\special{pa 3767 3881}%
\special{pa 3745 3871}%
\special{pa 3745 3847}%
\special{pa 3700 3900}%
\special{fp}%
%
\special{pn 8}%
\special{pa 2620 760}%
\special{pa 1420 1760}%
\special{fp}%
\special{sh 1}%
\special{pa 1420 1760}%
\special{pa 1484 1733}%
\special{pa 1461 1726}%
\special{pa 1458 1702}%
\special{pa 1420 1760}%
\special{fp}%
%
\special{pn 8}%
\special{pa 3720 760}%
\special{pa 5060 1150}%
\special{fp}%
\special{sh 1}%
\special{pa 5060 1150}%
\special{pa 5002 1112}%
\special{pa 5009 1135}%
\special{pa 4990 1151}%
\special{pa 5060 1150}%
\special{fp}%
%
\special{pn 8}%
\special{pa 5110 1940}%
\special{pa 5110 2340}%
\special{fp}%
\special{sh 1}%
\special{pa 5110 2340}%
\special{pa 5130 2273}%
\special{pa 5110 2287}%
\special{pa 5090 2273}%
\special{pa 5110 2340}%
\special{fp}%
%
\special{pn 8}%
\special{pa 5070 3120}%
\special{pa 4520 3750}%
\special{fp}%
\special{sh 1}%
\special{pa 4520 3750}%
\special{pa 4579 3713}%
\special{pa 4555 3710}%
\special{pa 4549 3687}%
\special{pa 4520 3750}%
\special{fp}%
%
\special{pn 8}%
\special{pa 1420 2520}%
\special{pa 1820 3720}%
\special{fp}%
\special{sh 1}%
\special{pa 1820 3720}%
\special{pa 1818 3650}%
\special{pa 1803 3669}%
\special{pa 1780 3663}%
\special{pa 1820 3720}%
\special{fp}%
\put(18.2000,-12.3000){\makebox(0,0)[lb]{$\mu_{1}$}}%
\put(13.4000,-32.1000){\makebox(0,0)[lb]{$\mu_{2}$}}%
\put(43.9000,-9.2000){\makebox(0,0)[lb]{$\mu_{2}$}}%
\put(51.5000,-22.3000){\makebox(0,0)[lb]{$\mu_{1}$}}%
\put(47.5000,-36.3000){\makebox(0,0)[lb]{$\mu_{2}$}}%
\end{picture}}\]
\end{example}
\begin{theorem}\rm{\cite{BDP}}
Let $\Lambda$ be a finite dimensional path algebra $kQ$ and denote by $\tsilt (\Lambda)$ the poset of two-term silting object of $\mathrm{D}^{\mathrm{b}}(\mod \Lambda)$ (refer to \cite{AIR} for definition).
Then the Hasse-quiver of $\tsilt (\Lambda)$ coincides to $\overrightarrow{\mathrm{E}}(Q)$. 
\end{theorem}
In \cite{AIR}, Adachi, Iyama and Reiten shows that $\tsilt(\Lambda)$ is isomorphic to the poset of support tilting modules $\stilt(\Lambda)$\;(see Subsection \ref{subsec:2.2} below for definition).
we have that the Hasse-quiver of $\stilt(\Lambda)$ coincides to $\overrightarrow{\mathrm{E}}(Q)$. We also remark that an oriented exchange graph
 $\overrightarrow{ \mathrm{E}}(Q)$ is realized by several posets coming from representation theory. For more detail, we refer to \cite{BDP}.
\subsection{Tilting modules}
\label{subsec:2.1}
In this subsection we recall the definition of tilting modules.
For a module $M$, we denote by $|M|$ the number of non-isomorphic indecomposable direct summand of $M$.

\begin{definition}
A $\Lambda$-module $M$ is said to be a \emph{partial tilting module} if it satisfies following conditions. 
\begin{enumerate}[(i)]
\item The projective dimension of $T$ is at most 1. 
\item $\Ext_{\Lambda}^{1}(T,T)=0.$
\end{enumerate} 

If partial tilting module $T$ satisfies $|M|=|\Lambda|$, then we call $T$ a \emph{tilting module}. The set of non-isomorphic basic tilting modules of $\Lambda$ is denoted by $\tilt(\Lambda)$
\end{definition}
For a $\Lambda$-module $M$, we put $M^{\perp_{1}}:=\{X\in \mod \Lambda\mid \Ext_{\Lambda}^{1}(M,X)=0\}.$
\begin{definitiontheorem}\cite{HU1}
Let $T_{1}$ and $T_{2}$ be two tilting modules. We write $T_{1}\leq T_{2}$ if
$T^{\perp_{1}}_{1}\subset T^{\perp_{1}}_{2}$. Then $\leq$ defines a partial order
on $\tilt(\Lambda)$. 
\end{definitiontheorem}
For a module $M$, we let \[\mathrm{Fac}M:=\{X\in \mod \Lambda\mid X\ \text{ is\ a\ factor\ module\ of\ finite\ direct\ sums\ of\ copies\ of\ M}\}.\]
  It is known that if $T$ is a tilting module, then $X$ is in $T^{\perp_{1}}$ if and only if $X$ is in $\mathrm{Fac}T$.
Therefore we have the following\;\cite{HU2}\;:
\[T\geq T^{'}\Leftrightarrow \Ext_{\Lambda}^{1}(T,T^{'})=0.\]
\subsection{Support $\tau$-tilting modules}
\label{subsec:2.2}
The notion of support $\tau$-tilting modules which were introduced in \cite{AIR}
is a generalization of that of tilting modules.

Let us recall the definition of support $\tau$-tilting modules.

\begin{definition}\cite{AIR}
Let $M$ be a $\Lambda$-module and $P$ be a projective $\Lambda$-module.
\begin{enumerate}[(1)]
\item $M$ is said to be a $\tau$-\emph{rigid module} if it satisfies
$\Hom_{\Lambda}(M, \tau M)=0$.
\item  $(M,P)$ is said to be a $\tau$-\emph{rigid pair} 
if $M$ is a $\tau$-rigid module and $\Hom_{\Lambda}(P,M)=0$.
\item $(M,P)$ is called a \emph{support $\tau$-tilting pair} 
if it is a $\tau$-rigid pair with $|M|+|P|=|\Lambda|$.
We then call $M$ a \emph{support $\tau$-tilting module}.
The set of non-isomorphic basic support $\tau$-tilting modules of $\Lambda$ is denoted by $\sttilt(\Lambda)$.
\end{enumerate}
\end{definition}
We note that if $M\in \sttilt(\Lambda)$, then there is a unique (up to isomorphism) basic projective module $P$ such that $(M,P)$ is a support $\tau$-tilting pair\;\cite{AIR}.
\begin{definitiontheorem}\cite{AIR}
Let $(M,P)$ and $(M^{'},P^{'})$ be two support $\tau$-tilting pair. We write $M\leq M^{'}$ if $\Hom_{\Lambda}(M,\tau M^{'})=0$ and $\add P^{'}\subset \add P $. 
Then $M\leq M'$ if and only if $\mathrm{Fac}M\subset \mathrm{Fac}M'$. Moreover, $\leq$ defines a partial order on $\sttilt(\Lambda)$.  
\end{definitiontheorem} 
We call $(N,U)$ an almost complete support $\tau$-tilting pair if $(N,U)$ is a $\tau$-rigid pair with $|N|+|U|=|\Lambda|-1$.
\begin{theorem}\rm{\cite{AIR}}
\begin{enumerate}[{\rm (1)}]
\item Let $(N,U)$ be a basic almost complete support $\tau$-tilting pair. Then $(N,U)$ is a direct summand of exactly two support $\tau$-tilting pairs.
\item Let $(M,P)$ and $(M^{'},P^{'})$ be two support $\tau$-tilting pair. Then there is an edge $M-M^{'}$ in the underlying graph of
 the Hasse quiver of $\sttilt(\Lambda)$ if and only if there exists basic almost complete support $\tau$-tilting pair $(N,U)$ such that $(N,U)$
  is a direct summand of $(M,P)$ and $(M^{'},P^{'})$.
\end{enumerate}   
\end{theorem}
For a basic $\tau$-rigid module $U$, we denote by $\sttilt_{U}(\Lambda):=\{T\in \sttilt(\Lambda)\mid U\in \add T\}$.
Then $\sttilt_{U}(\Lambda)$ has a maximum element $T_{U}$ \cite{AIR}. We call $T_{U}$ the Bongartz completion of $U$.   
\begin{theorem}\rm{\cite{J}}
\label{reduction theorem}
Let $U$ be a basic partial $\tau$-tilting module and let $T_{U}$ be the Bongartz completion of $U$. If we set $\Gamma_{U}:=\End_{\Lambda}(T_{U})/\langle e_{U}\rangle$, then
  $|\Gamma|=|\Lambda|-|U|$ and $\sttilt_{U}(\Lambda)\simeq \sttilt(\Gamma_{U})$, where $e_{U}$ is the idempotent corresponding to the projective $\End_{\Lambda}(T_{U})$-module $\Hom_{\Lambda}(T_{U},U)$.
  
\end{theorem}

\subsection{Hereditary case}
\label{subsec:2.3}
Let $Q$ be a finite connected acyclic quiver. We denote by $Q_{0}$\;(resp.\;$Q_{1}$) the set of vertices\;(resp.\;arrows) of $Q$. 
From now on, we assume that $\Lambda$ is a path algebra $kQ$. In this paper, for any paths
 $w:a_{0}\stackrel{\alpha_{1}}{\rightarrow}a_{1}\stackrel{\alpha_{2}}{\rightarrow}\cdots \stackrel{\alpha_{r}}{\rightarrow} a_{r}$
  and $w^{'}:b_{0}\stackrel{\beta_{1}}{\rightarrow}b_{1}\stackrel{\beta_{2}}{\rightarrow}\cdots \stackrel{\beta_{s}}{\rightarrow} b_{s}$ in $Q$,
   the product is defined by \[w\cdot w^{'}:=\left\{\begin{array}{ll}
 a_{0}\stackrel{\alpha_{1}}{\rightarrow}a_{1}\stackrel{\alpha_{2}}{\rightarrow}\cdots \stackrel{\alpha_{r}}{\rightarrow} a_{r}
 =b_{0}\stackrel{\beta_{1}}{\rightarrow}b_{1}\stackrel{\beta_{2}}{\rightarrow}\cdots \stackrel{\beta_{s}}{\rightarrow} b_{s} & \mathrm{if\ }a_{r}=b_{0} \\ 
 0 & \mathrm{if\ }a_{r}\neq b_{0},
 \end{array}\right.\]
 in $kQ$.
 For a module $M\in \mod \Lambda$, we denote by $Q(M)$ the full subquiver of $Q$
  with $Q(M)_{0}=\mathrm{supp}(M):=\{a\in Q_{0}\mid \dim Me_{a}>0\}$, where $e_{a}$ is the primitive idempotent of $\Lambda$ corresponding to $a\in Q_{0}$.
  By definition, we can regard $M$ as a sincere $kQ(M)$-module. 
\begin{definition}\normalfont\cite{AIR,IT}
A $\Lambda$-module $M$ is said to be a \emph{support tilting module} if $M$ is a tilting $kQ(M)$-module. The set of non-isomorphic basic support tilting module of $\Lambda$ is denoted by $\stilt(\Lambda)$.
\end{definition}
Since $\Lambda$ is a finite dimensional hereditary algebra, we have $\sttilt(\Lambda)=\stilt(\Lambda)\;(\mathrm{see}\ $\cite{AIR}) and the partial order on $\stilt(\Lambda)$ is defined as follows:
\[M\geq M^{'}\Leftrightarrow \Ext_{\Lambda}^{1}(M,M^{'})=0\ \mathrm{and}\ \mathrm{supp}(M')\subset \mathrm{supp}(M).\ \ (M,M^{'}\in \stilt(\Lambda))\]
\begin{theorem}\rm{\cite{J}}
\label{stilting reduction}
 Let $U$ be a basic partial tilting module and let $\Gamma_{U}$ be an algebra considered in Theorem\;\ref{reduction theorem}. Then $\Gamma_{U}$ is hereditary and we have  
$\stilt_{U}(\Lambda)\simeq \stilt(\Gamma_{U})$. 
\end{theorem}
Let $M\in \mod \Lambda$ and let $P$ be a projective $\Lambda$-module.
We set \[\Tau_{\Lambda}(M,P)=\Tau(M,P):=(\tau M\oplus \nu(P), M_{\mathrm{pr}} )\] where $\nu$ is the Nakayama functor and $M_{\mathrm{pr}}$ is a maximal projective direct summand. We also set \[\Tau^{-}(M,P):=(\tau^{-}M\oplus P, \nu^{-}M_{\mathrm{in}})\] where $M_{\mathrm{in}}$ is a maximal injective direct summand.
 Note that
\[\Tau\Tau^{-}(M,P)=(M,P)=\Tau^{-}\Tau(M,P).\]
\begin{lemma}
\label{1}\rm{\cite{AIR}}
$(M,P)$ is a support tilting pair if and only if $\Tau(M,P)$ is a support tilting pair. In particular $\Tau$ and $\Tau^{-}$ induces a graph automorphism
\[G(\stilt(\Lambda))\simeq G(\stilt(\Lambda)),\]
where $G(\stilt(\Lambda))$ is the underlying graph of the Hasse quiver of $\stilt(\Lambda)$.
\end{lemma}
\begin{example}
\label{exstilta3}  We give posets of support tilting modules for type $A_{3}$ quivers. \\
(1)\;Let $Q(1)=1\to 2 \to 3$ and $\Lambda(1)=kQ(1)$. Then $\stilt(\Lambda(1))$ is given by the following. 
 \[ 
{\unitlength 0.1in%
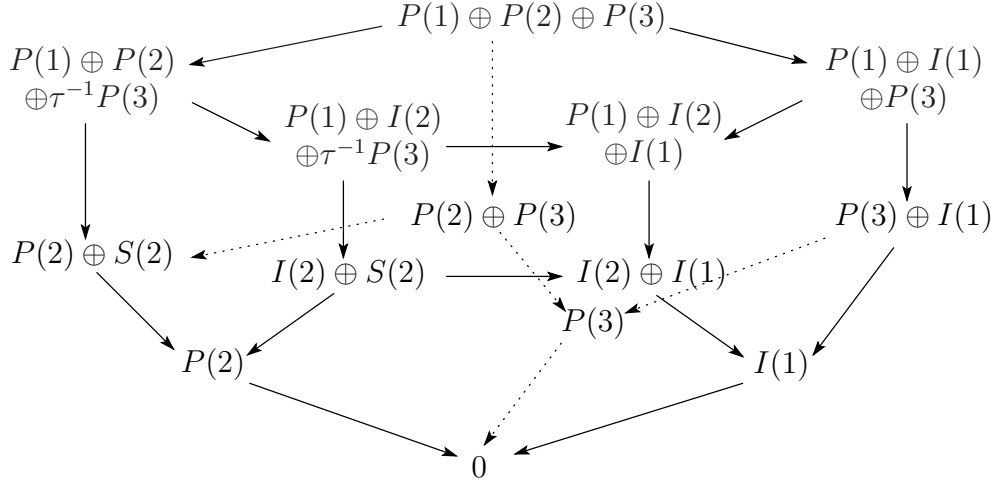
\begin{picture}( 47.9700, 23.4300)(  9.5000,-34.1300)%
\put(23.9000,-19.4000){\makebox(0,0)[lb]{$\begin{array}{c}P(1)\oplus I(2)\\ \oplus \tau^{-1}P(3)\end{array}$}}%
\put(30.5000,-12.0000){\makebox(0,0)[lb]{$P(1)\oplus P(2)\oplus P(3)$}}%
\put(9.5000,-16.4000){\makebox(0,0)[lb]{$\begin{array}{c}P(1)\oplus P(2)\\ \oplus \tau^{-1}P(3)\end{array}$}}%
\put(52.2000,-16.4000){\makebox(0,0)[lb]{$\begin{array}{c}P(1)\oplus I(1)\\ \oplus P(3)\end{array}$}}%
\put(38.6000,-19.3000){\makebox(0,0)[lb]{$\begin{array}{c}P(1)\oplus I(2)\\ \oplus I(1)\end{array}$}}%
%
\special{pn 8}%
\special{pa 2970 1160}%
\special{pa 1970 1360}%
\special{fp}%
\special{sh 1}%
\special{pa 1970 1360}%
\special{pa 2039 1367}%
\special{pa 2022 1350}%
\special{pa 2031 1327}%
\special{pa 1970 1360}%
\special{fp}%
%
\special{pn 8}%
\special{pa 4480 1170}%
\special{pa 5190 1350}%
\special{fp}%
\special{sh 1}%
\special{pa 5190 1350}%
\special{pa 5130 1314}%
\special{pa 5138 1337}%
\special{pa 5120 1353}%
\special{pa 5190 1350}%
\special{fp}%
%
\special{pn 8}%
\special{pa 1980 1560}%
\special{pa 2380 1760}%
\special{fp}%
\special{sh 1}%
\special{pa 2380 1760}%
\special{pa 2329 1712}%
\special{pa 2332 1736}%
\special{pa 2311 1748}%
\special{pa 2380 1760}%
\special{fp}%
%
\special{pn 8}%
\special{pa 3310 1790}%
\special{pa 3910 1790}%
\special{fp}%
\special{sh 1}%
\special{pa 3910 1790}%
\special{pa 3843 1770}%
\special{pa 3857 1790}%
\special{pa 3843 1810}%
\special{pa 3910 1790}%
\special{fp}%
%
\special{pn 8}%
\special{pa 5170 1550}%
\special{pa 4770 1750}%
\special{fp}%
\special{sh 1}%
\special{pa 4770 1750}%
\special{pa 4839 1738}%
\special{pa 4818 1726}%
\special{pa 4821 1702}%
\special{pa 4770 1750}%
\special{fp}%
\put(23.9000,-25.4000){\makebox(0,0)[lb]{$I(2)\oplus S(2)$}}%
\put(39.9000,-25.4000){\makebox(0,0)[lb]{$I(2)\oplus I(1)$}}%
%
\special{pn 8}%
\special{pa 2770 1970}%
\special{pa 2770 2370}%
\special{fp}%
\special{sh 1}%
\special{pa 2770 2370}%
\special{pa 2790 2303}%
\special{pa 2770 2317}%
\special{pa 2750 2303}%
\special{pa 2770 2370}%
\special{fp}%
%
\special{pn 8}%
\special{pa 4370 1970}%
\special{pa 4370 2370}%
\special{fp}%
\special{sh 1}%
\special{pa 4370 2370}%
\special{pa 4390 2303}%
\special{pa 4370 2317}%
\special{pa 4350 2303}%
\special{pa 4370 2370}%
\special{fp}%
\put(31.2000,-22.4000){\makebox(0,0)[lb]{$P(2)\oplus P(3)$}}%
%
\special{pn 8}%
\special{pa 2980 2170}%
\special{pa 1980 2370}%
\special{dt 0.045}%
\special{sh 1}%
\special{pa 1980 2370}%
\special{pa 2049 2377}%
\special{pa 2032 2360}%
\special{pa 2041 2337}%
\special{pa 1980 2370}%
\special{fp}%
%
\special{pn 8}%
\special{pa 3550 1240}%
\special{pa 3550 2040}%
\special{dt 0.045}%
\special{sh 1}%
\special{pa 3550 2040}%
\special{pa 3570 1973}%
\special{pa 3550 1987}%
\special{pa 3530 1973}%
\special{pa 3550 2040}%
\special{fp}%
\put(10.3000,-24.4000){\makebox(0,0)[lb]{$P(2)\oplus S(2)$}}%
%
\special{pn 8}%
\special{pa 1420 1670}%
\special{pa 1420 2270}%
\special{fp}%
\special{sh 1}%
\special{pa 1420 2270}%
\special{pa 1440 2203}%
\special{pa 1420 2217}%
\special{pa 1400 2203}%
\special{pa 1420 2270}%
\special{fp}%
\put(53.4000,-22.4000){\makebox(0,0)[lb]{$P(3)\oplus I(1)$}}%
%
\special{pn 8}%
\special{pa 5720 1670}%
\special{pa 5720 2070}%
\special{fp}%
\special{sh 1}%
\special{pa 5720 2070}%
\special{pa 5740 2003}%
\special{pa 5720 2017}%
\special{pa 5700 2003}%
\special{pa 5720 2070}%
\special{fp}%
\put(39.1000,-28.0000){\makebox(0,0)[lb]{$P(3)$}}%
\put(19.2000,-30.1000){\makebox(0,0)[lb]{$P(2)$}}%
%
\special{pn 8}%
\special{pa 1480 2450}%
\special{pa 1880 2850}%
\special{fp}%
\special{sh 1}%
\special{pa 1880 2850}%
\special{pa 1847 2789}%
\special{pa 1842 2812}%
\special{pa 1819 2817}%
\special{pa 1880 2850}%
\special{fp}%
%
\special{pn 8}%
\special{pa 2720 2560}%
\special{pa 2270 2880}%
\special{fp}%
\special{sh 1}%
\special{pa 2270 2880}%
\special{pa 2336 2858}%
\special{pa 2313 2849}%
\special{pa 2313 2825}%
\special{pa 2270 2880}%
\special{fp}%
\put(52.1000,-30.2000){\makebox(0,0)[rb]{$I(1)$}}%
%
\special{pn 8}%
\special{pa 4410 2570}%
\special{pa 4860 2890}%
\special{fp}%
\special{sh 1}%
\special{pa 4860 2890}%
\special{pa 4817 2835}%
\special{pa 4817 2859}%
\special{pa 4794 2868}%
\special{pa 4860 2890}%
\special{fp}%
%
\special{pn 8}%
\special{pa 5290 2270}%
\special{pa 4250 2660}%
\special{dt 0.045}%
\special{sh 1}%
\special{pa 4250 2660}%
\special{pa 4319 2655}%
\special{pa 4300 2641}%
\special{pa 4305 2618}%
\special{pa 4250 2660}%
\special{fp}%
\special{pa 4250 2660}%
\special{pa 4250 2660}%
\special{dt 0.045}%
%
\special{pn 8}%
\special{pa 5650 2320}%
\special{pa 5230 2880}%
\special{fp}%
\special{sh 1}%
\special{pa 5230 2880}%
\special{pa 5286 2839}%
\special{pa 5262 2837}%
\special{pa 5254 2815}%
\special{pa 5230 2880}%
\special{fp}%
\put(34.4000,-35.2000){\makebox(0,0)[lb]{$0$}}%
%
\special{pn 8}%
\special{pa 2280 3030}%
\special{pa 3360 3410}%
\special{fp}%
\special{sh 1}%
\special{pa 3360 3410}%
\special{pa 3304 3369}%
\special{pa 3310 3392}%
\special{pa 3290 3407}%
\special{pa 3360 3410}%
\special{fp}%
%
\special{pn 8}%
\special{pa 4860 3030}%
\special{pa 3670 3400}%
\special{fp}%
\special{sh 1}%
\special{pa 3670 3400}%
\special{pa 3740 3399}%
\special{pa 3721 3384}%
\special{pa 3728 3361}%
\special{pa 3670 3400}%
\special{fp}%
%
\special{pn 8}%
\special{pa 3910 2820}%
\special{pa 3510 3350}%
\special{dt 0.045}%
\special{sh 1}%
\special{pa 3510 3350}%
\special{pa 3566 3309}%
\special{pa 3542 3307}%
\special{pa 3534 3285}%
\special{pa 3510 3350}%
\special{fp}%
%
\special{pn 8}%
\special{pa 3590 2240}%
\special{pa 3930 2650}%
\special{dt 0.045}%
\special{sh 1}%
\special{pa 3930 2650}%
\special{pa 3903 2586}%
\special{pa 3896 2609}%
\special{pa 3872 2611}%
\special{pa 3930 2650}%
\special{fp}%
%
\special{pn 8}%
\special{pa 3310 2470}%
\special{pa 3910 2470}%
\special{fp}%
\special{sh 1}%
\special{pa 3910 2470}%
\special{pa 3843 2450}%
\special{pa 3857 2470}%
\special{pa 3843 2490}%
\special{pa 3910 2470}%
\special{fp}%
\end{picture}}\]
 (2)\;Let $Q(2)=1\rightarrow 2\leftarrow 3$ and $\Lambda(2)=kQ(2)$. Then $\stilt (\Lambda(2))$ is given by the following.
 \[
{\unitlength 0.1in%
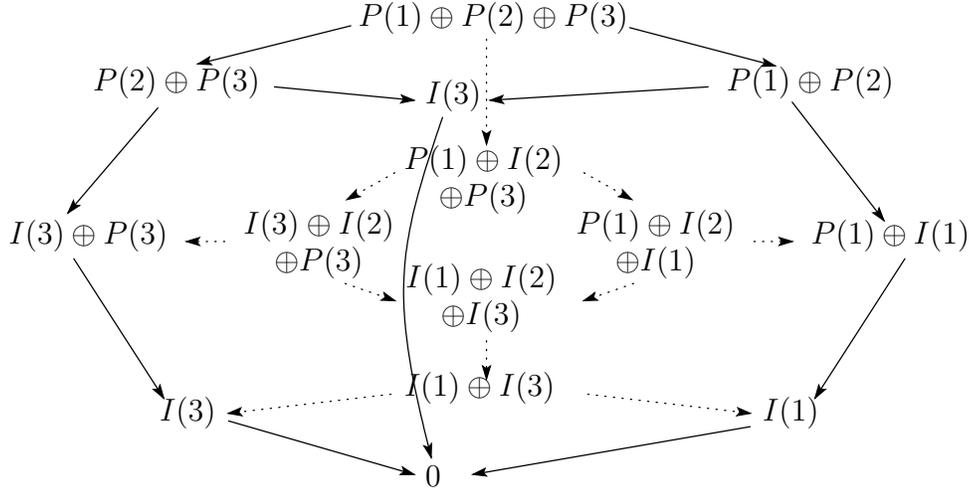
\begin{picture}( 46.9000, 24.5500)( 16.7000,-30.5500)%
%
\special{pn 8}%
\special{pa 4170 760}%
\special{pa 4170 1310}%
\special{dt 0.045}%
\special{sh 1}%
\special{pa 4170 1310}%
\special{pa 4190 1243}%
\special{pa 4170 1257}%
\special{pa 4150 1243}%
\special{pa 4170 1310}%
\special{fp}%
\put(35.0000,-7.3000){\makebox(0,0)[lb]{$P(1)\oplus P(2) \oplus P(3)$}}%
\put(36.7000,-16.9000){\makebox(0,0)[lb]{$\begin{array}{c}P(1)\oplus I(2)\\ \oplus P(3)\end{array}$}}%
\put(28.3000,-20.3000){\makebox(0,0)[lb]{$\begin{array}{c}I(3)\oplus I(2)\\ \oplus P(3)\end{array}$}}%
\put(36.8000,-23.1000){\makebox(0,0)[lb]{$\begin{array}{c}I(1)\oplus I(2)\\ \oplus I(3)\end{array}$}}%
\put(21.1000,-10.7000){\makebox(0,0)[lb]{$ P(2) \oplus P(3)$}}%
\put(16.7000,-18.7000){\makebox(0,0)[lb]{$ I(3) \oplus P(3)$}}%
%
\special{pn 8}%
\special{pa 2450 1120}%
\special{pa 1980 1670}%
\special{fp}%
\special{sh 1}%
\special{pa 1980 1670}%
\special{pa 2039 1632}%
\special{pa 2015 1629}%
\special{pa 2008 1606}%
\special{pa 1980 1670}%
\special{fp}%
\put(24.6000,-28.0000){\makebox(0,0)[lb]{$ I(3)$}}%
%
\special{pn 8}%
\special{pa 2010 1910}%
\special{pa 2480 2640}%
\special{fp}%
\special{sh 1}%
\special{pa 2480 2640}%
\special{pa 2461 2573}%
\special{pa 2451 2595}%
\special{pa 2427 2595}%
\special{pa 2480 2640}%
\special{fp}%
%
\special{pn 8}%
\special{pa 3690 1460}%
\special{pa 3440 1600}%
\special{dt 0.045}%
\special{sh 1}%
\special{pa 3440 1600}%
\special{pa 3508 1585}%
\special{pa 3487 1574}%
\special{pa 3488 1550}%
\special{pa 3440 1600}%
\special{fp}%
\put(37.4000,-26.7000){\makebox(0,0)[lb]{$ I(1) \oplus I(3)$}}%
%
\special{pn 8}%
\special{pa 3670 2620}%
\special{pa 2820 2720}%
\special{dt 0.045}%
\special{sh 1}%
\special{pa 2820 2720}%
\special{pa 2889 2732}%
\special{pa 2873 2714}%
\special{pa 2884 2692}%
\special{pa 2820 2720}%
\special{fp}%
%
\special{pn 8}%
\special{pa 2800 1820}%
\special{pa 2600 1820}%
\special{dt 0.045}%
\special{sh 1}%
\special{pa 2600 1820}%
\special{pa 2667 1840}%
\special{pa 2653 1820}%
\special{pa 2667 1800}%
\special{pa 2600 1820}%
\special{fp}%
%
\special{pn 8}%
\special{pa 3430 2040}%
\special{pa 3690 2130}%
\special{dt 0.045}%
\special{sh 1}%
\special{pa 3690 2130}%
\special{pa 3634 2089}%
\special{pa 3640 2113}%
\special{pa 3620 2127}%
\special{pa 3690 2130}%
\special{fp}%
\put(55.4000,-20.3000){\makebox(0,0)[rb]{$\begin{array}{c}P(1)\oplus I(2)\\ \oplus I(1)\end{array}$}}%
\put(63.0000,-10.7000){\makebox(0,0)[rb]{$ P(1) \oplus P(2) $}}%
\put(67.0000,-18.7000){\makebox(0,0)[rb]{$ P(1) \oplus I(1)$}}%
%
\special{pn 8}%
\special{pa 5770 1090}%
\special{pa 6250 1710}%
\special{fp}%
\special{sh 1}%
\special{pa 6250 1710}%
\special{pa 6225 1645}%
\special{pa 6217 1668}%
\special{pa 6193 1670}%
\special{pa 6250 1710}%
\special{fp}%
\put(59.1000,-28.0000){\makebox(0,0)[rb]{$ I(1)$}}%
%
\special{pn 8}%
\special{pa 6360 1910}%
\special{pa 5890 2640}%
\special{fp}%
\special{sh 1}%
\special{pa 5890 2640}%
\special{pa 5943 2595}%
\special{pa 5919 2595}%
\special{pa 5909 2573}%
\special{pa 5890 2640}%
\special{fp}%
%
\special{pn 8}%
\special{pa 4680 1460}%
\special{pa 4930 1600}%
\special{dt 0.045}%
\special{sh 1}%
\special{pa 4930 1600}%
\special{pa 4882 1550}%
\special{pa 4883 1574}%
\special{pa 4862 1585}%
\special{pa 4930 1600}%
\special{fp}%
%
\special{pn 8}%
\special{pa 4700 2620}%
\special{pa 5550 2720}%
\special{dt 0.045}%
\special{sh 1}%
\special{pa 5550 2720}%
\special{pa 5486 2692}%
\special{pa 5497 2714}%
\special{pa 5481 2732}%
\special{pa 5550 2720}%
\special{fp}%
%
\special{pn 8}%
\special{pa 5570 1820}%
\special{pa 5770 1820}%
\special{dt 0.045}%
\special{sh 1}%
\special{pa 5770 1820}%
\special{pa 5703 1800}%
\special{pa 5717 1820}%
\special{pa 5703 1840}%
\special{pa 5770 1820}%
\special{fp}%
%
\special{pn 8}%
\special{pa 4940 2040}%
\special{pa 4680 2130}%
\special{dt 0.045}%
\special{sh 1}%
\special{pa 4680 2130}%
\special{pa 4750 2127}%
\special{pa 4730 2113}%
\special{pa 4736 2089}%
\special{pa 4680 2130}%
\special{fp}%
\put(38.5000,-11.4000){\makebox(0,0)[lb]{$ I(3)$}}%
%
\special{pn 8}%
\special{pa 3060 1010}%
\special{pa 3790 1080}%
\special{fp}%
\special{sh 1}%
\special{pa 3790 1080}%
\special{pa 3726 1054}%
\special{pa 3737 1075}%
\special{pa 3722 1094}%
\special{pa 3790 1080}%
\special{fp}%
%
\special{pn 8}%
\special{pa 5330 1010}%
\special{pa 4190 1080}%
\special{fp}%
\special{sh 1}%
\special{pa 4190 1080}%
\special{pa 4258 1096}%
\special{pa 4243 1077}%
\special{pa 4255 1056}%
\special{pa 4190 1080}%
\special{fp}%
%
\special{pn 8}%
\special{pa 3460 690}%
\special{pa 2660 890}%
\special{fp}%
\special{sh 1}%
\special{pa 2660 890}%
\special{pa 2730 893}%
\special{pa 2712 877}%
\special{pa 2720 854}%
\special{pa 2660 890}%
\special{fp}%
%
\special{pn 8}%
\special{pa 4920 700}%
\special{pa 5680 900}%
\special{fp}%
\special{sh 1}%
\special{pa 5680 900}%
\special{pa 5621 864}%
\special{pa 5628 886}%
\special{pa 5610 902}%
\special{pa 5680 900}%
\special{fp}%
%
\special{pn 8}%
\special{pa 4170 2340}%
\special{pa 4170 2540}%
\special{dt 0.045}%
\special{sh 1}%
\special{pa 4170 2540}%
\special{pa 4190 2473}%
\special{pa 4170 2487}%
\special{pa 4150 2473}%
\special{pa 4170 2540}%
\special{fp}%
\put(38.5000,-31.0000){\makebox(0,0)[lb]{$ 0$}}%
%
\special{pn 8}%
\special{pa 3940 1170}%
\special{pa 3929 1201}%
\special{pa 3919 1232}%
\special{pa 3908 1263}%
\special{pa 3888 1325}%
\special{pa 3878 1355}%
\special{pa 3848 1448}%
\special{pa 3812 1572}%
\special{pa 3788 1665}%
\special{pa 3781 1696}%
\special{pa 3775 1727}%
\special{pa 3768 1759}%
\special{pa 3763 1790}%
\special{pa 3757 1821}%
\special{pa 3753 1852}%
\special{pa 3748 1883}%
\special{pa 3745 1915}%
\special{pa 3742 1946}%
\special{pa 3740 1977}%
\special{pa 3738 2009}%
\special{pa 3736 2071}%
\special{pa 3736 2103}%
\special{pa 3737 2134}%
\special{pa 3738 2166}%
\special{pa 3740 2197}%
\special{pa 3742 2229}%
\special{pa 3745 2261}%
\special{pa 3748 2292}%
\special{pa 3752 2324}%
\special{pa 3756 2355}%
\special{pa 3760 2387}%
\special{pa 3765 2419}%
\special{pa 3770 2450}%
\special{pa 3775 2482}%
\special{pa 3787 2546}%
\special{pa 3793 2577}%
\special{pa 3821 2705}%
\special{pa 3828 2736}%
\special{pa 3836 2768}%
\special{pa 3843 2800}%
\special{pa 3867 2896}%
\special{pa 3874 2928}%
\special{pa 3880 2950}%
\special{fp}%
%
\special{pn 8}%
\special{pa 3874 2928}%
\special{pa 3880 2950}%
\special{fp}%
\special{sh 1}%
\special{pa 3880 2950}%
\special{pa 3882 2880}%
\special{pa 3866 2899}%
\special{pa 3843 2891}%
\special{pa 3880 2950}%
\special{fp}%
%
\special{pn 8}%
\special{pa 2820 2760}%
\special{pa 3790 3040}%
\special{fp}%
\special{sh 1}%
\special{pa 3790 3040}%
\special{pa 3731 3002}%
\special{pa 3739 3025}%
\special{pa 3720 3041}%
\special{pa 3790 3040}%
\special{fp}%
%
\special{pn 8}%
\special{pa 5550 2790}%
\special{pa 4100 3040}%
\special{fp}%
\special{sh 1}%
\special{pa 4100 3040}%
\special{pa 4169 3048}%
\special{pa 4153 3031}%
\special{pa 4162 3009}%
\special{pa 4100 3040}%
\special{fp}%
\end{picture}}\]
 (3)\;Let $Q(3)=Q(2)^{\mathrm{op}}=1\leftarrow 2\rightarrow 3$ and $\Lambda(3)=kQ(3)$. Then $\stilt(\Lambda(3))\simeq \stilt(\Lambda(2))^{\mathrm{op}}. $

\end{example}  
\section{Some remarks on tilting modules of path algebras of type $A$ and $\tilde{A}_{n,1}$ }
\subsection{Ext-vanishing conditions for path algebras of type $A$}
Let $Q$ be a quiver of type $A_{n}$ and $\stackrel{1}{\circ}-\stackrel{2}{\circ}-\cdots-\stackrel{n}{\circ}$ be its underlying graph.

For $i\in Q_{0}\setminus \{n\}$ we set 
\[\mathrm{d}(i):=\left\{\begin{array}{ll}
+ & \mathrm{if}\ i\rightarrow i+1 \\ 
- & \mathrm{if}\ i\leftarrow i+1.
\end{array}\right. \]
From Gabriel's Theorem $(\dimvec):\mod kQ\rightarrow \mathbb{Z}^{Q_{0}}_{\geq 0}$ induces a bijection 
\[\mathrm{ind}\;kQ\stackrel{1:1}{\leftrightarrow}\{\mathrm{positive\ roots\ of}\ A_{n}\}. \] Now denote by $L(i,j)$\;($0\leq i<j\leq n$) the indecomposable module corresponding to $\alpha_{i+1}+\alpha_{i+2}+\cdots \alpha_{j}$ where $\alpha_{i}$ is a
simple root corresponding to $i\in Q_{0}$.
Let $L(i^{'},j^{'}):=\tau L(i,j)$
Then we can check that
\[i^{'}=\left\{\begin{array}{ll}
i+1 & \mathrm{if}\  i\ \mathrm{and}\ i+1\ \mathrm{are\ not\ sink\ and}\ \mathrm{d}(i)=+ \\ 
i-1 & \mathrm{if}\ i\ \mathrm{and}\ i+1\ \mathrm{are\ not\ sink\ and}\ \mathrm{d}(i)=- \\ 
\mathrm{max}\{a< i \mid a\ \mathrm{is\ source\ of}\ Q\}-1 & \mathrm{if}\ i\ \mathrm{is\ sink\ of}\ Q \\ 
\mathrm{min}\{a> i \mid a\ \mathrm{is\ source\ of}\ Q\} & \mathrm{if}\ i+1\ \mathrm{is\ sink\ of}\ Q.
\end{array}\right. \]
\[j^{'}=\left\{\begin{array}{ll}
j+1 & \mathrm{if}\  j\ \mathrm{and}\ j+1\ \mathrm{are\ not\ sink\ and}\ \mathrm{d}(j)=+ \\ 
j-1 & \mathrm{if}\ j\ \mathrm{and}\ j+1\ \mathrm{are\ not\ sink\ and}\ \mathrm{d}(j)=- \\ 
\mathrm{max}\{a< j \mid a\ \mathrm{is\ source\ of}\ Q\}-1 & \mathrm{if}\ j\ \mathrm{is\ sink\ of}\ Q \\ 
\mathrm{min}\{a> j \mid a\ \mathrm{is\ source\ of}\ Q\} & \mathrm{if}\ j+1\ \mathrm{is\ sink\ of}\ Q.
\end{array}\right. \]

\begin{lemma}
\label{hv}
Let $X,Y\in \mathrm{ind}\;kQ$ and $f:=(f_{a})_{a\in Q_{0}}\in \Hom_{kQ}(X,Y)$. Then $f=0$ if and only if $\mathrm{supp}(X)\cap \mathrm{supp}(Y)= \emptyset$
 or  $\exists a\in \mathrm{supp}(X)\cap \mathrm{supp}(Y)$ such that $f_{a}=0$.

\end{lemma}
\begin{proof}
Assume that $f=0$ and $\mathrm{supp}(X)\cap \mathrm{supp}(Y)\neq \emptyset$. Then there is a vertex $a\in \mathrm{supp}(X)\cap \mathrm{supp}(Y)$. It is obvious that $f_{a}=0$.
Note that $\mathrm{supp}(X)\cap \mathrm{supp}(Y)= \emptyset$ implies $f=0$. Hence we assume that $\exists a\in \mathrm{supp}(X)\cap \mathrm{supp}(Y)$ such that $f_{a}=0$.
Since the full sub quiver $Q'$ of $Q$ with $Q'_{0}=\mathrm{supp}(X)\cap \mathrm{supp}(Y)$ is connected, it is sufficient to check the following:
\[x\in Q'_{0}\mathrm{\ and}\ f_{x}=0\Rightarrow f_{y}=0\ \mathrm{for\ any\ neighbor}\ y\ \mathrm{of}\ x\ \mathrm{in\ }Q.\]
Let $x\in Q'_{0}$ such that $f_{x}=0$ and let $y$ be a neighbor of $x$. Denote by $\alpha\in Q_{1}$ the arrow between $x$ and $y$. We also denote by $X_{\alpha}$\;(resp.\;$Y_{\alpha}$) the linear map in $X$\;$(resp.\;Y)$ associated with $\alpha$. If $y\not\in Q'_{0}$, then it is obvious that $f_{y}=0$. Hence we may assume that $y\in Q'_{0}$.
Then we have that $X_{\alpha}$ and $Y_{\alpha}$ are isomorphisms. Since $f_{x}=0$, we have $f_{y}=0$.     

\end{proof}

\begin{lemma}
\label{ev}
$\Ext^{1}_{kQ}(L(i,j),L(k,l))=0=\Ext^{1}_{kQ}(L(k,l),L(i,j))$ if and only if one of the following conditions holds,
\begin{itemize}
\item[(1)]$[i,j]\cap [k,l]=\emptyset$,\\
\item[(2)]$i=k$ or $j=l$,\\
\item[(3)]$i<k<j<l$ and $\mathrm{d}(j)\neq \mathrm{d}(k)$,\\
\item[(3')]$k<i<l<j$ and $\mathrm{d}(i)\neq \mathrm{d}(l)$,\\
\item[(4)]$i<k<l<j$ and $\mathrm{d}(k)= \mathrm{d}(l)$,\\
\item[(4')]$k<i<j<l$ and $\mathrm{d}(i)= \mathrm{d}(j)$.
\end{itemize}
\end{lemma}
\begin{proof}
We consider the cases $(1),\;(2),\;(3),\;(3^{'}),\;(4),\;(4^{'})$ and $(5),\;(5^{'}),\;(6),\;(6^{'})$ where 
\begin{itemize}
\item[(5)]$i<k\leq j<l$ and $\mathrm{d}(j)= \mathrm{d}(k)$,\\
\item[(5')]$k<i\leq l<j$ and $\mathrm{d}(i)= \mathrm{d}(l)$,\\
\item[(6)]$i<k<l<j$ and $\mathrm{d}(k)\neq \mathrm{d}(l)$,\\
\item[(6')]$k<i<j<l$ and $\mathrm{d}(i)\neq \mathrm{d}(j)$.
\end{itemize}
Set $\mathrm{E}(X,Y):=\Ext^{1}_{kQ}(X,Y)\oplus \Ext^{1}_{kQ}(Y,X)$.  Then it is sufficient to show that
\[(c)\Rightarrow \left\{\begin{array}{ll}
E(L(i,j),L(k,l))=0 & \mathrm{if}\ c\in\{1,2,3,3^{'},4,4^{'}\} \\ 
E(L(i,j),L(k,l))\neq 0 & \mathrm{if}\ c\in\{5,5^{'},6,6^{'}\}
\end{array}\right.\]

($c=1$) In this case the assertion is obvious.

($c=2$) Let $L(i^{'},j^{'}):=\tau L(i,j)$ and $f=(f_{a})_{a\in Q_{0}}\in \Hom_{kQ}(L(k,l),L(i^{'},j^{'}))$. We consider the case $i=k$. First we assume $k=i<i^{'}$. Then we get $\mathrm{d}(i^{'})=+$. If $\mathrm{supp}(L(k,l))\cap \mathrm{supp}(L(i^{'},j^{'}))\neq \emptyset$, then it is easy to check that $i^{'}+1\in \mathrm{supp}(L(k,l))\cap \mathrm{supp}(L(i^{'},j^{'}))$. Moreover  $f_{i^{'}+1}=0$. In fact, we have following commutative diagram.
\[ 
{\unitlength 0.1in%
\begin{picture}( 11.0000,  9.8700)( 12.3000,-12.5700)%
\put(14.2000,-6.5000){\makebox(0,0)[lb]{$k$}}%
\put(22.4000,-6.7000){\makebox(0,0)[lb]{$k$}}%
\put(22.4000,-12.9000){\makebox(0,0)[lb]{$k$}}%
\put(14.2000,-13.0000){\makebox(0,0)[lb]{$0$}}%
\put(14.4000,-4.1000){\makebox(0,0)[lb]{$i'$}}%
\put(21.4000,-4.0000){\makebox(0,0)[lb]{$i'+1$}}%
%
\special{pn 8}%
\special{pa 1590 630}%
\special{pa 2190 630}%
\special{fp}%
\special{sh 1}%
\special{pa 2190 630}%
\special{pa 2123 610}%
\special{pa 2137 630}%
\special{pa 2123 650}%
\special{pa 2190 630}%
\special{fp}%
%
\special{pn 8}%
\special{pa 1590 1230}%
\special{pa 2190 1230}%
\special{fp}%
\special{sh 1}%
\special{pa 2190 1230}%
\special{pa 2123 1210}%
\special{pa 2137 1230}%
\special{pa 2123 1250}%
\special{pa 2190 1230}%
\special{fp}%
%
\special{pn 8}%
\special{pa 1470 690}%
\special{pa 1470 1140}%
\special{fp}%
\special{sh 1}%
\special{pa 1470 1140}%
\special{pa 1490 1073}%
\special{pa 1470 1087}%
\special{pa 1450 1073}%
\special{pa 1470 1140}%
\special{fp}%
%
\special{pn 8}%
\special{pa 2270 690}%
\special{pa 2270 1140}%
\special{fp}%
\special{sh 1}%
\special{pa 2270 1140}%
\special{pa 2290 1073}%
\special{pa 2270 1087}%
\special{pa 2250 1073}%
\special{pa 2270 1140}%
\special{fp}%
\put(12.3000,-9.6000){\makebox(0,0)[lb]{$f_{i'}$}}%
\put(23.3000,-9.5000){\makebox(0,0)[lb]{$f_{i'+1}$}}%
%
\special{pn 8}%
\special{pa 1720 560}%
\special{pa 1751 542}%
\special{pa 1781 530}%
\special{pa 1809 533}%
\special{pa 1837 550}%
\special{pa 1864 572}%
\special{pa 1891 588}%
\special{pa 1919 589}%
\special{pa 1948 574}%
\special{pa 1977 553}%
\special{pa 1980 550}%
\special{fp}%
%
\special{pn 8}%
\special{pa 1720 1160}%
\special{pa 1751 1142}%
\special{pa 1781 1130}%
\special{pa 1809 1133}%
\special{pa 1837 1150}%
\special{pa 1864 1172}%
\special{pa 1891 1188}%
\special{pa 1919 1189}%
\special{pa 1948 1174}%
\special{pa 1977 1153}%
\special{pa 1980 1150}%
\special{fp}%
\end{picture}}
\]
   By Lemma\;\ref{hv} we get $f=0$. Next we assume $i>i^{'}$.
Then we get $\mathrm{d}(i)=-$. If $\mathrm{supp}(L(k,l))\cap \mathrm{supp}(L(i^{'},j^{'}))\neq \emptyset$,
 then it is easy to check that $i+1=k+1\in \mathrm{supp}(L(k,l))\cap \mathrm{supp}(L(i^{'},j^{'}))$ and $f_{i+1}=0$. By Lemma\;\ref{hv} we get $f=0$.
Similarly, we can check that $f=0$ if $j=l$. Since the condition $(2)$ is symmetric, we also obtain $\Hom_{\Lambda}(L(i,j),\tau L(k,l))=0$.

($c=3$) Let $L(i^{'},j^{'}):=\tau L(i,j)$, $L(k^{'},l^{'}):=\tau L(k,l)$, 
 $f=(f_{a})_{a\in Q_{0}}\in \Hom_{kQ}(L(k,l),L(i^{'},j^{'}))$ and $g=(g_{a})_{a\in Q_{0}}\in \Hom_{kQ}(L(i,j),L(k^{'},l^{'}))$. First we show that $f=0$. 
 
 Suppose that $j<j^{'}$, then we get $\mathrm{d}(j)=+$ and $\mathrm{d}(k)=-$. In the case $i^{'}<k$, we have $k+1\in \mathrm{supp}(L(k,l))\cap\mathrm{supp}(L(i^{'},j^{'}))$ and $f_{k+1}=0$. This implies $f=0$. Therefore we may assume that $i^{'}\geq k > i$. In this case $\mathrm{d}(i^{'})=+$ and this shows $i^{'}>k$. Since $i<j\to j+1$, we have $i'\leq j<l$. Now it is easy to check that
 $i^{'}+1\in \mathrm{supp}(L(k,l))\cap\mathrm{supp}(L(i^{'},j^{'}))$ and $f_{i^{'}+1}=0$. This also implies $f=0$.
 
 If $j>j^{'}$ then we get $\mathrm{d}(j^{'})=-$.  We may assume that $\mathrm{supp}(L(k,l))\cap\mathrm{supp}(L(i^{'},j^{'}))\neq \emptyset$. Then $j^{'}\in \mathrm{supp}(L(k,l))\cap\mathrm{supp}(L(i^{'},j^{'}))$. Now is is easy to check that $f_{j^{'}}=0$. Hence, by Lemma\;\ref{hv}, we obtain $f=0$.
 
Note that $\mathrm{D}L(i,j)_{Q}=L(i,j)_{Q^{\mathrm{op}}}$ and $d_{Q}(x)=-d_{Q^{\mathrm{op}}}(x)$ for any $x\in Q_{0}=Q^{\mathrm{op}}_{0}$. Thus by using above result, we have
\[\Ext_{\Lambda}^{1}(L(k,l),L(i,j))\simeq \Ext_{\Lambda^{\mathrm{op}}}^{1}(\mathrm{D}L(i,j),\mathrm{D}L(k,l))=\Ext_{\Lambda^{\mathrm{op}}}^{1}(L(i,j)_{Q^{\mathrm{op}}},L(k,l)_{Q^{\mathrm{op}}})=0.\]
 
($c=4$) Let $L(i^{'},j^{'}):=\tau L(i,j)$, $L(k^{'},l^{'}):=\tau L(k,l)$, 
 $f=(f_{a})_{a\in Q_{0}}\in \Hom_{kQ}(L(k,l),L(i^{'},j^{'}))$ and $g=(g_{a})_{a\in Q_{0}}\in \Hom_{kQ}(L(i,j),L(k^{'},l^{'}))$. First we show $f=0$. Therefore we can assume $\mathrm{supp}(L(k,l))\cap\mathrm{supp}(L(i^{'},j^{'}))\neq \emptyset$. 
 
If $k<j^{'}<l$ then $j^{'}\in \mathrm{supp}(L(k,l))\cap\mathrm{supp}(L(i^{'},j^{'}))$ and $\mathrm{d}(j^{'})=-$. Now we can easily see that $f_{j^{'}}=0$. By Lemma\;\ref{hv}, we have
$f=0$.

If $j^{'}=l$ and $i^{'}<k$. Then $\mathrm{d}(k)=\mathrm{d}(l)=\mathrm{d}(j^{'})=-$ and $k+1\in \mathrm{supp}(L(k,l))\cap\mathrm{supp}(L(i^{'},j^{'}))$.
Now it is easy to check that $f_{k+1}=0$ and $f=0$.

If $j^{'}=l$ and $i^{'}\geq k$. Then $\mathrm{d}(k)=-$ and $\mathrm{d}(i^{'})=+$. This implies $i^{'}>k$ and $i^{'}+1\in \mathrm{supp}(L(k,l))\cap\mathrm{supp}(L(i^{'},j^{'}))$.
Now it is easy to check that $f_{i^{'+1}}=0$. By Lemma\;\ref{hv} we get $f=0$.

If $j^{'}>l$ and $\mathrm{d}(k)=\mathrm{d}(l)=+$. Then it is easy to check that
$l\in \mathrm{supp}(L(k,l))\cap\mathrm{supp}(L(i^{'},j^{'}))$ and $f_{l}=0$. Lemma\;\ref{hv} shows $f=0$.

If $j^{'}>l$, $\mathrm{d}(k)=\mathrm{d}(l)=-$ and $i^{'}<k$. Then we can easily
see that $k+1\in \mathrm{supp}(L(k,l))\cap\mathrm{supp}(L(i^{'},j^{'}))$ and
$f_{k+1}=0$ and $f=0$.

Therefore we may assume that $j^{'}>l$, $\mathrm{d}(k)=\mathrm{d}(l)=-$ and $i^{'}\geq k$. Then similar to the case $j^{'}=l$ and $i^{'}\geq k$, we have $f=0$.

By taking $k$-dual, we also have $g=0$.

($c=5$) In this case there is an non split exact sequence,
\[\left\{\begin{array}{ll}
0\rightarrow L(k,l)\rightarrow L(i,l)\oplus L(k,j)\rightarrow L(i,j)\rightarrow 0 & \mathrm{if}\ \mathrm{d}(j)=\mathrm{d}(k)=+  \\ 
0\rightarrow L(i,j)\rightarrow L(i,l)\oplus L(k,j)\rightarrow L(k,l)\rightarrow 0  & \mathrm{if}\ \mathrm{d}(j)=\mathrm{d}(k)=-.
\end{array}\right. \] 
In particular we get $E(L(i,j),L(k,l))\neq 0$.

($c=6$) Similar to the case $c=5$, we get $E(L(i,j),L(k,l))\neq 0.$

\end{proof}
\subsection{Some remarks on support tilting modules of $\tilde{A}_{n,1}$ }
In this subsection, we let $\Lambda=kQ$ where $Q$ is a quiver of type $\tilde{A}_{n,1}$, i.e., $Q$ is the following quiver:
\[
{\unitlength 0.1in%
\begin{picture}( 16.0000,  5.0200)( 16.6000, -8.2700)%
\put(16.6000,-8.6000){\makebox(0,0)[lb]{$0$}}%
%
\special{pn 8}%
\special{pa 1800 800}%
\special{pa 2100 800}%
\special{fp}%
\special{sh 1}%
\special{pa 2100 800}%
\special{pa 2033 780}%
\special{pa 2047 800}%
\special{pa 2033 820}%
\special{pa 2100 800}%
\special{fp}%
%
\special{pn 4}%
\special{sh 1}%
\special{ar 2170 800 8 8 0  6.28318530717959E+0000}%
\special{sh 1}%
\special{ar 2370 800 8 8 0  6.28318530717959E+0000}%
\special{sh 1}%
\special{ar 2570 800 8 8 0  6.28318530717959E+0000}%
\special{sh 1}%
\special{ar 2770 800 8 8 0  6.28318530717959E+0000}%
%
\special{pn 8}%
\special{pa 2900 800}%
\special{pa 3200 800}%
\special{fp}%
\special{sh 1}%
\special{pa 3200 800}%
\special{pa 3133 780}%
\special{pa 3147 800}%
\special{pa 3133 820}%
\special{pa 3200 800}%
\special{fp}%
\put(32.6000,-8.6000){\makebox(0,0)[lb]{$n$}}%
%
\special{pn 8}%
\special{pa 1760 730}%
\special{pa 1781 706}%
\special{pa 1803 683}%
\special{pa 1824 659}%
\special{pa 1846 636}%
\special{pa 1891 591}%
\special{pa 1915 569}%
\special{pa 1939 548}%
\special{pa 1964 527}%
\special{pa 1990 507}%
\special{pa 2017 487}%
\special{pa 2044 469}%
\special{pa 2073 451}%
\special{pa 2101 434}%
\special{pa 2131 418}%
\special{pa 2161 404}%
\special{pa 2192 390}%
\special{pa 2222 377}%
\special{pa 2254 366}%
\special{pa 2285 356}%
\special{pa 2317 347}%
\special{pa 2349 340}%
\special{pa 2381 334}%
\special{pa 2413 330}%
\special{pa 2445 327}%
\special{pa 2477 325}%
\special{pa 2509 326}%
\special{pa 2541 328}%
\special{pa 2572 332}%
\special{pa 2604 337}%
\special{pa 2634 345}%
\special{pa 2665 354}%
\special{pa 2695 364}%
\special{pa 2725 376}%
\special{pa 2754 389}%
\special{pa 2783 403}%
\special{pa 2812 418}%
\special{pa 2841 435}%
\special{pa 2869 452}%
\special{pa 2923 488}%
\special{pa 2950 508}%
\special{pa 2976 527}%
\special{pa 3002 548}%
\special{pa 3028 568}%
\special{pa 3053 589}%
\special{pa 3077 609}%
\special{pa 3125 651}%
\special{pa 3149 673}%
\special{pa 3173 694}%
\special{pa 3180 700}%
\special{fp}%
%
\special{pn 8}%
\special{pa 3173 694}%
\special{pa 3180 700}%
\special{fp}%
\special{sh 1}%
\special{pa 3180 700}%
\special{pa 3142 641}%
\special{pa 3140 665}%
\special{pa 3116 672}%
\special{pa 3180 700}%
\special{fp}%
\end{picture}}\]
\begin{lemma}[See \cite{KT}, Lemma\;6.1]
\label{dimup}
Let $\Delta$ be a acyclic quiver. Assume that for any vertex $a\in \Delta_{0}$, we have
$\#\{\alpha\in \Delta_{1}\mid s(\alpha)=a\ \mathrm{or\ }t(\alpha)=a\}\geq 2.$ Then any irreducible morphisms between indecomposable pre-projective
modules of $k\Delta$ are injective. Dually any irreducible morphism between indecomposable pre-injective modules are surjective.
\end{lemma}
\begin{remark}
\label{prin}
Lemma\;\ref{dimup} implies that $\tau^{<0}P(i)$ is sincere for any $i\in Q_{0}$. We also have that $\tau^{>0}I(i)$ is sincere for any $i\in Q_{0}$. 
In fact, for any non-projective indecomposable pre-projective\;(resp. non-injective indecomposable pre-injective)  module $M$, there is a path from $P(0)$ to $M$\;(resp. $M$ to $I(n)$) in the Auslander-Reiten quiver of $\mod \Lambda$ and 
$P(0)$\;(resp. $I(n)$) is sincere. This shows that there is no support tilting module which contains both a non-zero pre-projective direct summand and a non-zero pre-injective direct summand.
\end{remark} 
\begin{lemma}
\label{l1}
Let $(M,P)$ be a support-tilting pair in $\mod \Lambda$. Then there exists
$r\in \mathbb{Z}$ such that $(M^{'},P^{'}):=\Tau^{r}\;(M,P)$ with $P^{'}\neq 0.$

\end{lemma}
\begin{proof}
We can assume $P=0$ and in this case $M$ is a tilting module of $\Lambda$. Since $\Lambda$ is a tame algebra, $M$ has either a pre-projective direct summand or a pre-injective direct summand (\cite[Lemma\;3.1]{HR}). If $\tau^{-r}\;P(i)\in \mathrm{add}\;M$, then $(0,P(i))\in \mathrm{add}\;\Tau^{r+1}\;(M,P)$.  If $\tau^{r}I(i)\in \mathrm{add}\;M$, then $(0,P(i))\in \mathrm{add}\;\Tau^{-r-1}\;(M,P).$ 
\end{proof}

\begin{proposition}
\label{representative}
Let $(M,P)$ be a support-tilting pair in $\mod \Lambda$. Then there is a unique $(r,i)\in \mathbb{Z}\times Q_{0}$ such that $\Tau^{r}\;(M,P)\in \mathrm{s\text{-}tilt}_{X_{i}}(\Lambda)$, where \[X_{i}:=\left\{\begin{array}{ll}
(P(i+1),P(i)) &\mathrm{if}\ 0\leq i\leq n-1  \\ 
(0,P(1)\oplus P(n)) & \mathrm{if}\ i=n. 
\end{array}\right. \]
\end{proposition}
\begin{proof}

(Existence) By Lemma\;\ref{l1} we can assume $P:=\oplus_{i\in I}P(i)\neq 0$. Let $a:=\mathrm{min}\;I$ and $b:=\mathrm{max}\;I$.
In the case $(a,b)=(0,n)$ it is obvious that $(M,P)\in \mathrm{s\text{-}tilt}_{X_{n}}(\Lambda)$. Thus we can assume $(a,b)\neq (0,n)$.
If $a\neq 0,\;b=n$ then there is a direct summand $N$ of $M$ such that $N\in \tilt(k(0\rightarrow \cdots \rightarrow a-1))$ and so $I(a-1)\in \mathrm{add}\;N\subset \mathrm{add}\;M$. This implies $\Tau^{-1}\;(M,P)\in \mathrm{s\text{-}tilt}_{X_{a-1}}(\Lambda)$. 
If $a=0,\;b\neq n$, then similarly we have $P(b+1)\in \mathrm{add}\;M$. This implies $(M,P)\in \mathrm{s\text{-}tilt}_{X_{b}}(\Lambda)$.

Now we assume that $a\neq 0$ and $b\neq n$. In this case there is a direct summand $N$ of $M$ such that $N\in \tilt (k(b+1\rightarrow\cdots \rightarrow n \leftarrow 0\rightarrow \cdots\rightarrow a-1))$.
 First we show that $P(b+1)\in \mathrm{add}\;N$ or $I(a-1)\in \mathrm{add}\;N$.
Let $\Lambda':=k(b+1\rightarrow\cdots \rightarrow n\leftarrow 0 \rightarrow \cdots a-1)$.  By Lemma\;\ref{ev}, we get $\Ext^{1}_{\Lambda'}(I(i),P(j))\neq 0$\;$(0\leq \forall i \leq a-1,\;b+1 \leq \forall j \leq n)$ and 
\begin{itemize}
\item $\mathrm{ind}\Ker\Ext_{\Lambda'}^{1}(-,P(b+1))=\mathrm{ind}\Lambda'\setminus \{I(i)\mid 0\leq  i \leq a-1\}.$
\item $\mathrm{ind}\Ker\Ext_{\Lambda'}^{1}(I(a-1),-)=\mathrm{ind}\Lambda'\setminus \{P(j)\mid b+1\leq  j \leq n\}.$
\end{itemize}
This implies that $N$ has either $P(b+1)$ or $I(a-1)$ as a direct summand. If $P(b+1)\in \mathrm{add}\;N$ then $(M,P)\in \mathrm{s\text{-}tilt}_{X_{b}}(\Lambda)$. If $I(a-1)\in \mathrm{add}\; N$ then
$\Tau^{-1}\;(M,P)\in \mathrm{s\text{-}tilt}_{X_{a-1}}(\Lambda)$.

(Uniqueness) Let $(M,P)\in \mathrm{s\text{-}tilt}_{X_{i}}(\Lambda)$. We note that if $i\neq j$ then 
 \[\mathrm{s\text{-}tilt}_{X_{i}}(\Lambda)\cap \mathrm{s\text{-}tilt}_{X_{j}}(\Lambda)=\emptyset.\] Therefore it is sufficient to show that 
\[\Tau^{r}(M,P)\notin \mathrm{s\text{-}tilt}_{X_{j}}(\Lambda)\ \ \forall (r,j)\in \mathbb{Z}_{<0}\times Q_{0}\cdots (\ast).\]
 We claim that $M$ has no 
 pre-injective direct summand. We first assume that $i\leq n-1$. Since $M$ has non-zero projective direct summand, $M$ has no non-zero pre-injective direct summand.\;(Remark\;\ref{prin})  
Second we consider the case $i=n$.  It is easy to check that $(\dimvec\;N)_{0}\geq 1$ for any
 pre-injective module $N$. Therefore $M$ has no pre-injective direct summand.
 
 Now we prove $(\ast)$. If $\Tau^{r}\;(M,P)\in \mathrm{s\text{-}tilt}_{X_{j}}(\Lambda)$ for some $(r,j)\in \mathbb{Z}_{<0}\times Q_{0}$. Then $\tau^{-r-1}I(j)\in \mathrm{add}\;M$
 and this is a contradiction.      
 
\end{proof}

\begin{example}
Let $n=2$. Then we have 
\[\begin{array}{lll}
\stilt_{X_{0}}(\Lambda)&=&\{(P(1)\oplus P(2),P(0)),(P(1)\oplus S(1),P(0))\} \\
\stilt_{X_{1}}(\Lambda)&=&\{(\tau S(1)\oplus P(2),P(1)),(P(2),P(0)\oplus P(1))\} \\
\stilt_{X_{2}}(\Lambda)&=&\{( S(1), P(0)\oplus P(2)),(0,P(0)\oplus P(1)\oplus P(2))\}
\end{array}.
\]
Since $\tau^{2}S(1)=S(1)$, following gives a complete list of basic support tilting modules (up to isomorphism) of $\Lambda$.
\[\begin{array}{ll}
& \{\tau^{-r}\Lambda,\;\tau^{-r}(\Lambda/P(2))\oplus \tau^{-r-1} P(2),\; \tau^{-r}P(0)\oplus \tau^{-r-1}(\Lambda /P(0)) \}\\ \vspace{3pt}
               \sqcup &\{S(1)\oplus \tau^{-r}(\Lambda / P(0)),\; S(1)\oplus \tau^{-r}P(0)\oplus \tau^{-r-1}P(1),\; S(1)\oplus \tau^{-r}(\Lambda/ P(1)) \}_{r\in \N_{\mathrm{odd}} }\\ \vspace{3pt}
               \sqcup &\{\tau S(1) \oplus \tau^{-r}(\Lambda / P(0)),\; \tau S(1)\oplus \tau^{-r}P(0)\oplus \tau^{-r-1}P(1),\; \tau S(1)\oplus \tau^{-r}(\Lambda/ P(1)) \}_{r\in \N_{\mathrm{even}}}\\ \vspace{3pt}
               \sqcup & \{\tau^{r}I_{\Lambda},\;\tau^{r}(I_{\Lambda}/I(0))\oplus \tau^{r+1} I(0),\; \tau^{r}I(2)\oplus \tau^{r+1}(I_{\Lambda} /I(2)) \}\\ \vspace{3pt}
               \sqcup &\{S(1)\oplus \tau^{r}I(0)\oplus \tau^{r-1}I(1),\; S(1)\oplus \tau^{r+1}I(1)\oplus \tau^{r}I(2),\; S(1)\oplus \tau^{r}(I_{\Lambda}/ I(1)) \}_{r\in \N_{\mathrm{odd}} }\\ \vspace{3pt}
               \sqcup &\{\tau S(1) \oplus \tau^{r+1}I(0)\oplus \tau^{r}I(1),\; \tau S(1)\oplus \tau^{r+1}I(1)\oplus \tau^{r}I(2),\; \tau S(1)\oplus \tau^{r}(I_{\Lambda}/ I(1)) \}_{r\in \N_{\mathrm{even}}}\\ \vspace{3pt}
               \sqcup &\{P(1)\oplus P(2),\; P(1)\oplus S(1),\; \tau S(1)\oplus P(2),\; \tau S(1)\oplus I(0),\; I(1)\oplus S(1),\;I(0)\oplus I(1)\}\\ \vspace{3pt}
               \sqcup &\{ P(2), S(1), I(0)\}\\ \vspace{3pt}
               \sqcup &\{0\},

\end{array}               
\] where $I_{\Lambda}=I(0)\oplus I(1) \oplus I(2)$ is a basic injective tilting module. 
Then the Hasse-quiver of $\stilt(\Lambda)$ is given by following.
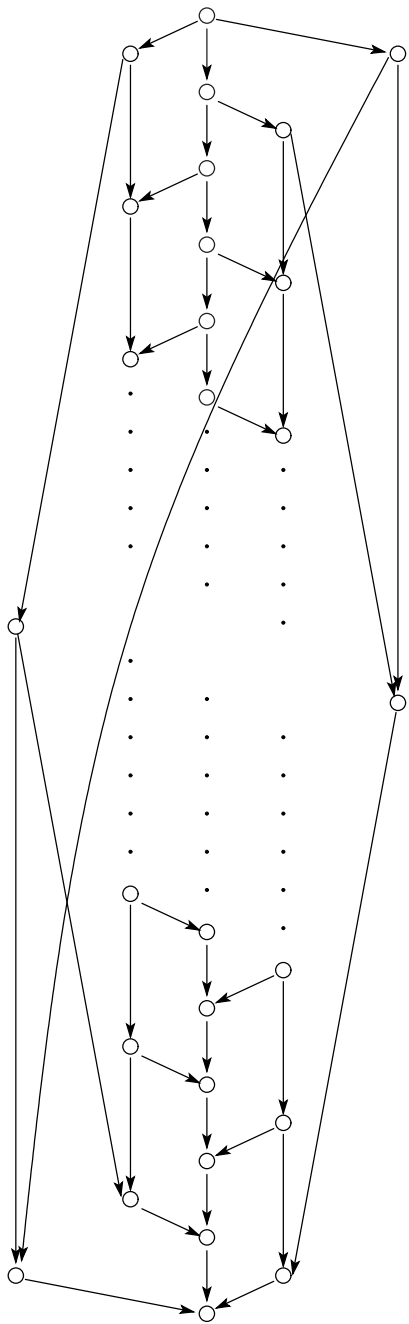
\[
{\unitlength 0.1in%
\begin{picture}( 20.8000, 68.8000)( 25.6000,-74.4000)%
%
\special{pn 8}%
\special{ar 3600 600 40 40  0.0000000  6.2831853}%
%
\special{pn 8}%
\special{ar 3600 1000 40 40  0.0000000  6.2831853}%
%
\special{pn 8}%
\special{ar 3600 1400 40 40  0.0000000  6.2831853}%
%
\special{pn 8}%
\special{ar 3600 1800 40 40  0.0000000  6.2831853}%
%
\special{pn 8}%
\special{ar 3600 2200 40 40  0.0000000  6.2831853}%
%
\special{pn 8}%
\special{ar 3600 2600 40 40  0.0000000  6.2831853}%
%
\special{pn 8}%
\special{pa 3600 670}%
\special{pa 3600 930}%
\special{fp}%
\special{sh 1}%
\special{pa 3600 930}%
\special{pa 3620 863}%
\special{pa 3600 877}%
\special{pa 3580 863}%
\special{pa 3600 930}%
\special{fp}%
%
\special{pn 8}%
\special{pa 3600 1070}%
\special{pa 3600 1330}%
\special{fp}%
\special{sh 1}%
\special{pa 3600 1330}%
\special{pa 3620 1263}%
\special{pa 3600 1277}%
\special{pa 3580 1263}%
\special{pa 3600 1330}%
\special{fp}%
%
\special{pn 8}%
\special{pa 3600 1470}%
\special{pa 3600 1730}%
\special{fp}%
\special{sh 1}%
\special{pa 3600 1730}%
\special{pa 3620 1663}%
\special{pa 3600 1677}%
\special{pa 3580 1663}%
\special{pa 3600 1730}%
\special{fp}%
%
\special{pn 8}%
\special{pa 3600 1870}%
\special{pa 3600 2130}%
\special{fp}%
\special{sh 1}%
\special{pa 3600 2130}%
\special{pa 3620 2063}%
\special{pa 3600 2077}%
\special{pa 3580 2063}%
\special{pa 3600 2130}%
\special{fp}%
%
\special{pn 8}%
\special{pa 3600 2270}%
\special{pa 3600 2530}%
\special{fp}%
\special{sh 1}%
\special{pa 3600 2530}%
\special{pa 3620 2463}%
\special{pa 3600 2477}%
\special{pa 3580 2463}%
\special{pa 3600 2530}%
\special{fp}%
%
\special{pn 8}%
\special{ar 3600 5400 40 40  0.0000000  6.2831853}%
%
\special{pn 8}%
\special{ar 3600 5800 40 40  0.0000000  6.2831853}%
%
\special{pn 8}%
\special{ar 3600 6200 40 40  0.0000000  6.2831853}%
%
\special{pn 8}%
\special{ar 3600 6600 40 40  0.0000000  6.2831853}%
%
\special{pn 8}%
\special{ar 3600 7000 40 40  0.0000000  6.2831853}%
%
\special{pn 8}%
\special{ar 3600 7400 40 40  0.0000000  6.2831853}%
%
\special{pn 8}%
\special{pa 3600 5470}%
\special{pa 3600 5730}%
\special{fp}%
\special{sh 1}%
\special{pa 3600 5730}%
\special{pa 3620 5663}%
\special{pa 3600 5677}%
\special{pa 3580 5663}%
\special{pa 3600 5730}%
\special{fp}%
%
\special{pn 8}%
\special{pa 3600 5870}%
\special{pa 3600 6130}%
\special{fp}%
\special{sh 1}%
\special{pa 3600 6130}%
\special{pa 3620 6063}%
\special{pa 3600 6077}%
\special{pa 3580 6063}%
\special{pa 3600 6130}%
\special{fp}%
%
\special{pn 8}%
\special{pa 3600 6270}%
\special{pa 3600 6530}%
\special{fp}%
\special{sh 1}%
\special{pa 3600 6530}%
\special{pa 3620 6463}%
\special{pa 3600 6477}%
\special{pa 3580 6463}%
\special{pa 3600 6530}%
\special{fp}%
%
\special{pn 8}%
\special{pa 3600 6670}%
\special{pa 3600 6930}%
\special{fp}%
\special{sh 1}%
\special{pa 3600 6930}%
\special{pa 3620 6863}%
\special{pa 3600 6877}%
\special{pa 3580 6863}%
\special{pa 3600 6930}%
\special{fp}%
%
\special{pn 8}%
\special{pa 3600 7070}%
\special{pa 3600 7330}%
\special{fp}%
\special{sh 1}%
\special{pa 3600 7330}%
\special{pa 3620 7263}%
\special{pa 3600 7277}%
\special{pa 3580 7263}%
\special{pa 3600 7330}%
\special{fp}%
%
\special{pn 8}%
\special{ar 4000 1200 40 40  0.0000000  6.2831853}%
%
\special{pn 8}%
\special{ar 4000 2000 40 40  0.0000000  6.2831853}%
%
\special{pn 8}%
\special{ar 4000 2800 40 40  0.0000000  6.2831853}%
%
\special{pn 8}%
\special{ar 3200 800 40 40  0.0000000  6.2831853}%
%
\special{pn 8}%
\special{ar 3200 1600 40 40  0.0000000  6.2831853}%
%
\special{pn 8}%
\special{ar 3200 2400 40 40  0.0000000  6.2831853}%
%
\special{pn 8}%
\special{pa 3550 630}%
\special{pa 3250 770}%
\special{fp}%
\special{sh 1}%
\special{pa 3250 770}%
\special{pa 3319 760}%
\special{pa 3298 747}%
\special{pa 3302 724}%
\special{pa 3250 770}%
\special{fp}%
%
\special{pn 8}%
\special{pa 3550 1430}%
\special{pa 3250 1570}%
\special{fp}%
\special{sh 1}%
\special{pa 3250 1570}%
\special{pa 3319 1560}%
\special{pa 3298 1547}%
\special{pa 3302 1524}%
\special{pa 3250 1570}%
\special{fp}%
%
\special{pn 8}%
\special{pa 3550 2230}%
\special{pa 3250 2370}%
\special{fp}%
\special{sh 1}%
\special{pa 3250 2370}%
\special{pa 3319 2360}%
\special{pa 3298 2347}%
\special{pa 3302 2324}%
\special{pa 3250 2370}%
\special{fp}%
%
\special{pn 8}%
\special{pa 3200 860}%
\special{pa 3200 1550}%
\special{fp}%
\special{sh 1}%
\special{pa 3200 1550}%
\special{pa 3220 1483}%
\special{pa 3200 1497}%
\special{pa 3180 1483}%
\special{pa 3200 1550}%
\special{fp}%
%
\special{pn 8}%
\special{pa 3200 1660}%
\special{pa 3200 2350}%
\special{fp}%
\special{sh 1}%
\special{pa 3200 2350}%
\special{pa 3220 2283}%
\special{pa 3200 2297}%
\special{pa 3180 2283}%
\special{pa 3200 2350}%
\special{fp}%
%
\special{pn 8}%
\special{pa 4000 1260}%
\special{pa 4000 1950}%
\special{fp}%
\special{sh 1}%
\special{pa 4000 1950}%
\special{pa 4020 1883}%
\special{pa 4000 1897}%
\special{pa 3980 1883}%
\special{pa 4000 1950}%
\special{fp}%
%
\special{pn 8}%
\special{pa 4000 2060}%
\special{pa 4000 2750}%
\special{fp}%
\special{sh 1}%
\special{pa 4000 2750}%
\special{pa 4020 2683}%
\special{pa 4000 2697}%
\special{pa 3980 2683}%
\special{pa 4000 2750}%
\special{fp}%
%
\special{pn 8}%
\special{pa 3660 1050}%
\special{pa 3960 1190}%
\special{fp}%
\special{sh 1}%
\special{pa 3960 1190}%
\special{pa 3908 1144}%
\special{pa 3912 1167}%
\special{pa 3891 1180}%
\special{pa 3960 1190}%
\special{fp}%
%
\special{pn 8}%
\special{pa 3660 1850}%
\special{pa 3960 1990}%
\special{fp}%
\special{sh 1}%
\special{pa 3960 1990}%
\special{pa 3908 1944}%
\special{pa 3912 1967}%
\special{pa 3891 1980}%
\special{pa 3960 1990}%
\special{fp}%
%
\special{pn 8}%
\special{pa 3660 2650}%
\special{pa 3960 2790}%
\special{fp}%
\special{sh 1}%
\special{pa 3960 2790}%
\special{pa 3908 2744}%
\special{pa 3912 2767}%
\special{pa 3891 2780}%
\special{pa 3960 2790}%
\special{fp}%
%
\special{pn 8}%
\special{ar 4000 5600 40 40  0.0000000  6.2831853}%
%
\special{pn 8}%
\special{ar 4000 6400 40 40  0.0000000  6.2831853}%
%
\special{pn 8}%
\special{ar 4000 7200 40 40  0.0000000  6.2831853}%
%
\special{pn 8}%
\special{pa 4000 5660}%
\special{pa 4000 6350}%
\special{fp}%
\special{sh 1}%
\special{pa 4000 6350}%
\special{pa 4020 6283}%
\special{pa 4000 6297}%
\special{pa 3980 6283}%
\special{pa 4000 6350}%
\special{fp}%
%
\special{pn 8}%
\special{pa 4000 6460}%
\special{pa 4000 7150}%
\special{fp}%
\special{sh 1}%
\special{pa 4000 7150}%
\special{pa 4020 7083}%
\special{pa 4000 7097}%
\special{pa 3980 7083}%
\special{pa 4000 7150}%
\special{fp}%
%
\special{pn 8}%
\special{ar 3200 5200 40 40  0.0000000  6.2831853}%
%
\special{pn 8}%
\special{ar 3200 6000 40 40  0.0000000  6.2831853}%
%
\special{pn 8}%
\special{ar 3200 6800 40 40  0.0000000  6.2831853}%
%
\special{pn 8}%
\special{pa 3200 5260}%
\special{pa 3200 5950}%
\special{fp}%
\special{sh 1}%
\special{pa 3200 5950}%
\special{pa 3220 5883}%
\special{pa 3200 5897}%
\special{pa 3180 5883}%
\special{pa 3200 5950}%
\special{fp}%
%
\special{pn 8}%
\special{pa 3200 6060}%
\special{pa 3200 6750}%
\special{fp}%
\special{sh 1}%
\special{pa 3200 6750}%
\special{pa 3220 6683}%
\special{pa 3200 6697}%
\special{pa 3180 6683}%
\special{pa 3200 6750}%
\special{fp}%
%
\special{pn 8}%
\special{pa 3260 5250}%
\special{pa 3560 5390}%
\special{fp}%
\special{sh 1}%
\special{pa 3560 5390}%
\special{pa 3508 5344}%
\special{pa 3512 5367}%
\special{pa 3491 5380}%
\special{pa 3560 5390}%
\special{fp}%
%
\special{pn 8}%
\special{pa 3260 6050}%
\special{pa 3560 6190}%
\special{fp}%
\special{sh 1}%
\special{pa 3560 6190}%
\special{pa 3508 6144}%
\special{pa 3512 6167}%
\special{pa 3491 6180}%
\special{pa 3560 6190}%
\special{fp}%
%
\special{pn 8}%
\special{pa 3260 6850}%
\special{pa 3560 6990}%
\special{fp}%
\special{sh 1}%
\special{pa 3560 6990}%
\special{pa 3508 6944}%
\special{pa 3512 6967}%
\special{pa 3491 6980}%
\special{pa 3560 6990}%
\special{fp}%
%
\special{pn 8}%
\special{pa 3950 5630}%
\special{pa 3650 5770}%
\special{fp}%
\special{sh 1}%
\special{pa 3650 5770}%
\special{pa 3719 5760}%
\special{pa 3698 5747}%
\special{pa 3702 5724}%
\special{pa 3650 5770}%
\special{fp}%
%
\special{pn 8}%
\special{pa 3950 6430}%
\special{pa 3650 6570}%
\special{fp}%
\special{sh 1}%
\special{pa 3650 6570}%
\special{pa 3719 6560}%
\special{pa 3698 6547}%
\special{pa 3702 6524}%
\special{pa 3650 6570}%
\special{fp}%
%
\special{pn 8}%
\special{pa 3950 7230}%
\special{pa 3650 7370}%
\special{fp}%
\special{sh 1}%
\special{pa 3650 7370}%
\special{pa 3719 7360}%
\special{pa 3698 7347}%
\special{pa 3702 7324}%
\special{pa 3650 7370}%
\special{fp}%
%
\special{pn 8}%
\special{ar 2600 3800 40 40  0.0000000  6.2831853}%
%
\special{pn 8}%
\special{ar 4600 4200 40 40  0.0000000  6.2831853}%
%
\special{pn 8}%
\special{pa 3160 830}%
\special{pa 2620 3770}%
\special{fp}%
\special{sh 1}%
\special{pa 2620 3770}%
\special{pa 2652 3708}%
\special{pa 2630 3718}%
\special{pa 2612 3701}%
\special{pa 2620 3770}%
\special{fp}%
%
\special{pn 8}%
\special{pa 2610 3840}%
\special{pa 3150 6780}%
\special{fp}%
\special{sh 1}%
\special{pa 3150 6780}%
\special{pa 3158 6711}%
\special{pa 3140 6728}%
\special{pa 3118 6718}%
\special{pa 3150 6780}%
\special{fp}%
%
\special{pn 8}%
\special{pa 4040 1220}%
\special{pa 4580 4160}%
\special{fp}%
\special{sh 1}%
\special{pa 4580 4160}%
\special{pa 4588 4091}%
\special{pa 4570 4108}%
\special{pa 4548 4098}%
\special{pa 4580 4160}%
\special{fp}%
%
\special{pn 8}%
\special{pa 4590 4250}%
\special{pa 4050 7190}%
\special{fp}%
\special{sh 1}%
\special{pa 4050 7190}%
\special{pa 4082 7128}%
\special{pa 4060 7138}%
\special{pa 4042 7121}%
\special{pa 4050 7190}%
\special{fp}%
%
\special{pn 8}%
\special{ar 4600 800 40 40  0.0000000  6.2831853}%
%
\special{pn 8}%
\special{pa 3650 620}%
\special{pa 4530 790}%
\special{fp}%
\special{sh 1}%
\special{pa 4530 790}%
\special{pa 4468 758}%
\special{pa 4478 780}%
\special{pa 4461 797}%
\special{pa 4530 790}%
\special{fp}%
%
\special{pn 8}%
\special{pa 4600 860}%
\special{pa 4600 4130}%
\special{fp}%
\special{sh 1}%
\special{pa 4600 4130}%
\special{pa 4620 4063}%
\special{pa 4600 4077}%
\special{pa 4580 4063}%
\special{pa 4600 4130}%
\special{fp}%
%
\special{pn 8}%
\special{pa 2600 3860}%
\special{pa 2600 7130}%
\special{fp}%
\special{sh 1}%
\special{pa 2600 7130}%
\special{pa 2620 7063}%
\special{pa 2600 7077}%
\special{pa 2580 7063}%
\special{pa 2600 7130}%
\special{fp}%
%
\special{pn 8}%
\special{pa 2650 7220}%
\special{pa 3530 7390}%
\special{fp}%
\special{sh 1}%
\special{pa 3530 7390}%
\special{pa 3468 7358}%
\special{pa 3478 7380}%
\special{pa 3461 7397}%
\special{pa 3530 7390}%
\special{fp}%
%
\special{pn 8}%
\special{ar 2600 7200 40 40  0.0000000  6.2831853}%
%
\special{pn 8}%
\special{pa 4550 820}%
\special{pa 4535 849}%
\special{pa 4520 877}%
\special{pa 4505 906}%
\special{pa 4490 934}%
\special{pa 4474 963}%
\special{pa 4459 991}%
\special{pa 4429 1049}%
\special{pa 4414 1077}%
\special{pa 4399 1106}%
\special{pa 4384 1134}%
\special{pa 4369 1163}%
\special{pa 4354 1191}%
\special{pa 4324 1249}%
\special{pa 4309 1277}%
\special{pa 4294 1306}%
\special{pa 4279 1334}%
\special{pa 4249 1392}%
\special{pa 4234 1420}%
\special{pa 4204 1478}%
\special{pa 4190 1506}%
\special{pa 4160 1564}%
\special{pa 4145 1592}%
\special{pa 4131 1621}%
\special{pa 4116 1650}%
\special{pa 4101 1678}%
\special{pa 4087 1707}%
\special{pa 4072 1736}%
\special{pa 4058 1764}%
\special{pa 4043 1793}%
\special{pa 4029 1822}%
\special{pa 4014 1851}%
\special{pa 4000 1880}%
\special{pa 3985 1908}%
\special{pa 3943 1995}%
\special{pa 3928 2024}%
\special{pa 3914 2053}%
\special{pa 3900 2081}%
\special{pa 3844 2197}%
\special{pa 3831 2226}%
\special{pa 3789 2313}%
\special{pa 3776 2342}%
\special{pa 3762 2371}%
\special{pa 3749 2400}%
\special{pa 3735 2429}%
\special{pa 3709 2487}%
\special{pa 3695 2517}%
\special{pa 3643 2633}%
\special{pa 3630 2663}%
\special{pa 3617 2692}%
\special{pa 3605 2721}%
\special{pa 3592 2750}%
\special{pa 3579 2780}%
\special{pa 3567 2809}%
\special{pa 3554 2838}%
\special{pa 3542 2868}%
\special{pa 3530 2897}%
\special{pa 3517 2927}%
\special{pa 3505 2956}%
\special{pa 3493 2986}%
\special{pa 3481 3015}%
\special{pa 3469 3045}%
\special{pa 3457 3074}%
\special{pa 3446 3104}%
\special{pa 3434 3134}%
\special{pa 3423 3163}%
\special{pa 3411 3193}%
\special{pa 3400 3223}%
\special{pa 3388 3252}%
\special{pa 3333 3402}%
\special{pa 3323 3431}%
\special{pa 3312 3461}%
\special{pa 3302 3491}%
\special{pa 3291 3521}%
\special{pa 3271 3581}%
\special{pa 3261 3612}%
\special{pa 3221 3732}%
\special{pa 3212 3763}%
\special{pa 3202 3793}%
\special{pa 3193 3823}%
\special{pa 3184 3854}%
\special{pa 3166 3914}%
\special{pa 3157 3945}%
\special{pa 3148 3975}%
\special{pa 3139 4006}%
\special{pa 3131 4036}%
\special{pa 3122 4067}%
\special{pa 3114 4098}%
\special{pa 3105 4128}%
\special{pa 3081 4221}%
\special{pa 3073 4251}%
\special{pa 3066 4282}%
\special{pa 3050 4344}%
\special{pa 3043 4375}%
\special{pa 3035 4406}%
\special{pa 2986 4623}%
\special{pa 2980 4654}%
\special{pa 2973 4685}%
\special{pa 2966 4717}%
\special{pa 2954 4779}%
\special{pa 2947 4810}%
\special{pa 2941 4842}%
\special{pa 2929 4904}%
\special{pa 2923 4936}%
\special{pa 2911 4998}%
\special{pa 2905 5030}%
\special{pa 2900 5061}%
\special{pa 2894 5093}%
\special{pa 2888 5124}%
\special{pa 2883 5156}%
\special{pa 2878 5187}%
\special{pa 2872 5219}%
\special{pa 2867 5250}%
\special{pa 2862 5282}%
\special{pa 2856 5314}%
\special{pa 2851 5345}%
\special{pa 2841 5409}%
\special{pa 2836 5440}%
\special{pa 2831 5472}%
\special{pa 2827 5504}%
\special{pa 2822 5535}%
\special{pa 2812 5599}%
\special{pa 2808 5631}%
\special{pa 2803 5662}%
\special{pa 2799 5694}%
\special{pa 2794 5726}%
\special{pa 2790 5758}%
\special{pa 2785 5790}%
\special{pa 2777 5854}%
\special{pa 2773 5885}%
\special{pa 2768 5917}%
\special{pa 2724 6269}%
\special{pa 2721 6301}%
\special{pa 2709 6397}%
\special{pa 2706 6429}%
\special{pa 2698 6493}%
\special{pa 2695 6526}%
\special{pa 2691 6558}%
\special{pa 2688 6590}%
\special{pa 2680 6654}%
\special{pa 2677 6686}%
\special{pa 2673 6718}%
\special{pa 2670 6750}%
\special{pa 2666 6782}%
\special{pa 2663 6815}%
\special{pa 2659 6847}%
\special{pa 2656 6879}%
\special{pa 2652 6911}%
\special{pa 2646 6975}%
\special{pa 2642 7007}%
\special{pa 2639 7040}%
\special{pa 2635 7072}%
\special{pa 2632 7104}%
\special{pa 2630 7120}%
\special{fp}%
%
\special{pn 8}%
\special{pa 2632 7104}%
\special{pa 2630 7120}%
\special{fp}%
\special{sh 1}%
\special{pa 2630 7120}%
\special{pa 2658 7056}%
\special{pa 2637 7067}%
\special{pa 2618 7051}%
\special{pa 2630 7120}%
\special{fp}%
%
\special{pn 4}%
\special{sh 1}%
\special{ar 3200 2580 8 8 0  6.28318530717959E+0000}%
\special{sh 1}%
\special{ar 3200 2780 8 8 0  6.28318530717959E+0000}%
\special{sh 1}%
\special{ar 3200 2980 8 8 0  6.28318530717959E+0000}%
\special{sh 1}%
\special{ar 3200 3180 8 8 0  6.28318530717959E+0000}%
\special{sh 1}%
\special{ar 3200 3380 8 8 0  6.28318530717959E+0000}%
\special{sh 1}%
\special{ar 3200 3980 8 8 0  6.28318530717959E+0000}%
\special{sh 1}%
\special{ar 3200 4180 8 8 0  6.28318530717959E+0000}%
\special{sh 1}%
\special{ar 3200 4380 8 8 0  6.28318530717959E+0000}%
\special{sh 1}%
\special{ar 3200 4580 8 8 0  6.28318530717959E+0000}%
\special{sh 1}%
\special{ar 3200 4780 8 8 0  6.28318530717959E+0000}%
\special{sh 1}%
\special{ar 3200 4980 8 8 0  6.28318530717959E+0000}%
%
\special{pn 4}%
\special{sh 1}%
\special{ar 4000 2980 8 8 0  6.28318530717959E+0000}%
\special{sh 1}%
\special{ar 4000 3180 8 8 0  6.28318530717959E+0000}%
\special{sh 1}%
\special{ar 4000 3380 8 8 0  6.28318530717959E+0000}%
\special{sh 1}%
\special{ar 4000 3580 8 8 0  6.28318530717959E+0000}%
\special{sh 1}%
\special{ar 4000 3780 8 8 0  6.28318530717959E+0000}%
\special{sh 1}%
\special{ar 4000 4380 8 8 0  6.28318530717959E+0000}%
\special{sh 1}%
\special{ar 4000 4580 8 8 0  6.28318530717959E+0000}%
\special{sh 1}%
\special{ar 4000 4780 8 8 0  6.28318530717959E+0000}%
\special{sh 1}%
\special{ar 4000 4980 8 8 0  6.28318530717959E+0000}%
\special{sh 1}%
\special{ar 4000 5180 8 8 0  6.28318530717959E+0000}%
\special{sh 1}%
\special{ar 4000 5380 8 8 0  6.28318530717959E+0000}%
%
\special{pn 4}%
\special{sh 1}%
\special{ar 3600 2780 8 8 0  6.28318530717959E+0000}%
\special{sh 1}%
\special{ar 3600 2980 8 8 0  6.28318530717959E+0000}%
\special{sh 1}%
\special{ar 3600 3180 8 8 0  6.28318530717959E+0000}%
\special{sh 1}%
\special{ar 3600 3380 8 8 0  6.28318530717959E+0000}%
\special{sh 1}%
\special{ar 3600 3580 8 8 0  6.28318530717959E+0000}%
\special{sh 1}%
\special{ar 3600 4180 8 8 0  6.28318530717959E+0000}%
\special{sh 1}%
\special{ar 3600 4380 8 8 0  6.28318530717959E+0000}%
\special{sh 1}%
\special{ar 3600 4580 8 8 0  6.28318530717959E+0000}%
\special{sh 1}%
\special{ar 3600 4780 8 8 0  6.28318530717959E+0000}%
\special{sh 1}%
\special{ar 3600 4980 8 8 0  6.28318530717959E+0000}%
\special{sh 1}%
\special{ar 3600 5180 8 8 0  6.28318530717959E+0000}%
\end{picture}}%
\]
\end{example}
\section{Slice tilting modules for path algebras}
Let $Q$ be a connected acyclic quiver.
 For any arrow $\alpha:x\rightarrow y$ in $Q$, denote by $\alpha^{-1}:y\rightarrow x$ its formal inverse.
  For any walk \[w:x_{0}\stackrel{\alpha_{1}^{\epsilon_{1}}}{\longrightarrow}x_{1}\stackrel{\alpha_{2}^{\epsilon_{2}}}{\longrightarrow}\cdots \stackrel{\alpha_{r}^{\epsilon_{r}}}{\longrightarrow}x_{r}
\ (\alpha_{i}\in Q_{1},\;\epsilon_{i}\in\{\pm1\},\;\forall i\in\{1,\cdots,r\})\] 
  on $Q$, we put $c^{+}(w):=\#\{i\mid \epsilon_{i}=+1 \}$. Then we set \[l_{Q}(x,y):=\left\{\begin{array}{cl}
  \mathrm{min}\{c^{+}(w)\mid w:\mathrm{walk\ from}\ x\ \mathrm{to}\ y\ \mathrm{on}\ Q\}& x\neq y\\
  0& x=y\\
\end{array}\right. .\]
\begin{definition}
Let $Q$ be a connected acyclic quiver. Then the translation quiver $\Z Q $ is defined as follows.
\begin{itemize}
\item $(\Z Q)_{0}:=\Z\times Q_{0}.$
\item $(\Z Q)_{1}:=\{(r,\alpha):(r,x)\to (r,y)\mid Q_{1}\ni \alpha:x\to y \}\sqcup\{(r,\alpha^{-}):(r,x)\to (r+1,y)\mid Q_{1}\ni \alpha:y\to x\}$.
\end{itemize}
\end{definition}
It is well-known that the pre-projective component of the Auslander-Reiten quiver of $\mod \Lambda$ is viewed as a full subquiver of the translation quiver $\Z_{\geq 0} Q^{\mathrm{op}}$ via the assignment $\tau^{-r}P(x)\mapsto (r,x)$. 
More precisely if $Q$ is not a Dynkin quiver, then above assignment is induces a quiver isomorphism. If $Q$ is a Dynkin quiver, then above assignment is extended as a quiver isomorphism
 $\Gamma(\mathrm{D}^{\mathrm{b}}(\mod \Lambda)) \simeq \Z Q^{\mathrm{op}}$, where $\mathrm{D}^{\mathrm{b}}(\mod \Lambda)$ is the bounded derived category of $\mod \Lambda$ and $\Gamma(\mathrm{D}^{\mathrm{b}}(\mod \Lambda))$ is its Auslander-Reiten quiver.
\begin{lemma}
\label{chain}
Let $Q$ be a connected acyclic quiver. Then  there is a path from $(r,x)$ to $(s,y)$ on the translation quiver $\Z Q^{\mathrm{op}}$ if and only if
$s\geq r+l_{Q}(x,y)$. 
\end{lemma}
\begin{proof}
Assume that there is a path
\[W:(r,x)=(r_{0},x_{0})\to\cdots\to (r_{l},x_{l})=(s,y).\]
By definition of  $(\Z Q^{\mathrm{op}})_{1}$, there is a walk 
\[w:x=x_{0}-x_{1}-\cdots-x_{l}=y\]
on $Q$.
Note that there is an arrow $(t,a)\rightarrow (u,b)$ in  $(\Z Q^{\mathrm{op}})_{1}$ if and only if
one of the following holds:
\begin{itemize}
\item[(a)]$a\leftarrow b$ in $Q$ and $t=u$.
\item[(b)]$a\rightarrow b$ in $Q$ and $t+1=u$.
\end{itemize}
Therefore we have that
\[s=r+c^{+}(w)\geq r+l_{Q}(x,y).\]
Conversely, we assume that $s\geq r+l_{Q}(x,y)$.
Let $w:x=x_{0}-\cdots -x_{l}=y$ be a walk on $Q$ such that $c^{+}(w)=l_{Q}(x,y)$.
Then we have a path
\[(r,x)=(r_{0},x_{0})\to\cdots\to (r_{l},x_{l})=(r_{l},y).\]
Then it is easy to check that $r_{l}-r_{0}=c^{+}(w)=l_{Q}(x,y)$.
We conclude that there exists a path from $(r,x)$ to $(r+l_{Q}(x,y),y)$.
Since $s\geq r+l_{Q}(x,y)$, we have a path from $(r+l_{Q}(x,y),y)$ to $(s,y)$. In particular,
there is a path from $(r,x)$ to $(s,y)$.
\end{proof}
\begin{definition}
Let $A$ be a finite dimensional algebra over $k$ and let $\mathcal{S}$ be a full subquiver of
the Auslander-Reiten quiver $\Gamma(\mod A)$ of $\mod A$. Then $\mathcal{S}$ said to be a slice
if $\mathcal{S}$ satisfies the following.
\begin{itemize}
\item[(S1)] $\add \mathcal{S}_{0} $ contains a sincere module $M\in \mod A$.
\item[(S2)] Let $X,Y\in \mathcal{S}_{0}$ and let  $W$ be a path from $X$ to $Y$ on $\Gamma(\mod A)$. Then for any $Z$ in $W$
is in $\mathcal{S}_{0}$.
\item[(S3)] Let $X$ be an indecomposable module. Then at most one of $X$, $\tau^{-}X$ is in $\mathcal{S}_{0}$.
\item[(S4)] Let $X\to Y$ be an arrow in $\Gamma(\mod A)$. If $Y\in \mathcal{S}_{0}$, then either $X$ or $\tau^{-}X$ is in $\mathcal{S}_{0}$. If $X\in \mathcal{S}_{0}$, then either $Y$ or $\tau Y$ is in $\mathcal{S}_{0}$.
\end{itemize} 
\end{definition}
Let $\mathcal{S}$ be a slice of a finite dimensional algebra $A$. Then $\mathcal{S}_{0}$ is finite and $T_{\mathcal{S}}:=\bigoplus_{X\in \mathcal{S}_{0}}X$
is a tilting module such that $\End_{A}(T_{\mathcal{S}})$ is hereditary (see \cite[Section\;4.2]{R}). We call $T_{\mathcal{S}}$ a slice tilting module. 
\begin{lemma}
\label{chain-condition}We have the following.
\begin{itemize}
\item[(1)] There is a chain of irreducible maps from $\tau^{-r}P(x)$ to
$\tau^{-s}P(y)$ if and only if 
\[s\geq r+l_{Q}(x,y).\]
\item[(2)] Let $T=\bigoplus_{x=1}^{n}\tau^{-r_{x}}P(x)$ be a slice tilting module. Then there is a chain of irreducible maps
from $\tau^{-r_{x}}P(x)$ to $\tau^{-r_{y}}P(y)$ if and only if $r_{y}=r_{x}+l_{Q}(x,y)$.
\item[(3)] $T=\bigoplus_{x=1}^{n}\tau^{-r_{x}}P(x)$ is a slice tilting module if and only if $r_{y}\leq r_{x}+l_{Q}(x,y)$ for any $x,y$. 
\end{itemize}
\end{lemma}
\begin{proof}
(1) If $Q$ is non-Dynkin, then the pre-projective component of $\Gamma(\mod \Lambda)$ is the translation quiver $\Z_{\geq 0} Q^{\mathrm{op}}$.
If $Q$ is Dynkin, then the Auslander-Reiten quiver of $\mathrm{D}^{\mathrm{b}}(\mod \Lambda)$ is the translation quiver $\Z Q^{\mathrm{op}}$. Note that for any pair of integers $i>j$ and for any $X,Y\in \ind \Lambda$, there is no path from
 $X[i]$ to $Y[j]$ where we denote by $X[i]$ the i-th shift of $X$ in $\mathrm{D}^{\mathrm{b}}(\mod \Lambda).$ Now the assertion follows from  Lemma\;\ref{chain}.

(2)
Let $w$ be a walk $w:x=x_{0}\stackrel{\alpha_{1}^{\epsilon_{1}}}{\longrightarrow}x_{1}\stackrel{\alpha_{2}^{\epsilon_{2}}}{\longrightarrow}\cdots \stackrel{\alpha_{l}^{\epsilon_{l}}}{\longrightarrow}x_{l}=y$ from $x$ to $y$ 
such that $l_{Q}(x,y)=c^{+}(w)$.

Since $T$ is a slice tilting module, there is a walk 
\[\tau^{-r_{x}}P(x)=\tau^{-r_{x_{0}}}P(x_{0})-\tau^{-r_{x_{1}}}P(x_{1})-\cdots - \tau^{-r_{x_{l}}}P(x_{l})=\tau^{-r_{y}}P(y).\]
We also have a path
\[(r,x)=(r_{x_{0}},x_{0})\to(s_{x_{1}},x_{1})\to\cdots \to (s_{x_{l}},x_{l})=(s_{y},y)\]
on $\Z Q^{\mathrm{op}}$ associated with $w$.
Then, by definition of $(\Z Q)_{1}$, we have $s_{x_{p}}\geq r_{x_{p}}$\;($\forall p$). 
Note that there is an arrow $(t,a)\rightarrow (u,b)$ if and only if
one of the following holds:
\begin{itemize}
\item[(a)]$a\leftarrow b$ in $Q$ and $t=u$.
\item[(b)]$a\rightarrow b$ in $Q$ and $t+1=u$.
\end{itemize}
Therefore we have \[ r_{y}\leq s_{y}=r_{x}+c_{+}(w)=r_{x}+l_{Q}(x,y).\]
We remark that above inequality implies that if $T$ is a slice tilting module, then $r_{y}\leq r_{x}+l_{Q}(x,y)$ for any $x,y$. 
If there is a chain of irreducible maps from $\tau^{-r_{x}}P(x)$ to $\tau^{-r_{y}}P(y)$, we have
\[r_{y}\geq r_{x}+l_{Q}(x,y).\]
Thus $r_{y}=r_{x}+l_{Q}(x,y)$.
On the other hand if $r_{y}=r_{x}+l_{Q}(x,y)$, then we have $r_{y}=s_{y}$. In particular, there is a chain of irreducible maps from $\tau^{-r_{x}}P(x)$ to $\tau^{-r_{y}}P(y)$.

(3) We already checked that if $T$ is a slice tilting module, then $r_{y}\leq r_{x}+l_{Q}(x,y)$ for any $x,y$. Thus we assume that $r_{y}\leq r_{x}+l_{Q}(x,y)$ for any $x,y$.
By the Auslander-Reiten duality, we have
 \[\dim\Ext_{\Lambda}(\tau^{-r_{y}}P(y),\tau^{-r_{x}}P(x))=\dim\Hom_{\Lambda}(\tau^{-r_{x}}P(x),\tau^{-r_{y}+1}P(y)). \]
Since $r_{y}-1<r_{x}+l_{Q}(x,y)$, there is no path from $\tau^{-r_{x}}P(x)$ to $\tau^{-r_{y}+1}P(y)$. In particular, we obtain that $\Hom_{\Lambda}(\tau^{-r_{x}}P(x),\tau^{-r_{y}+1}P(y))=0$.
We conclude that $T$ is a tilting module. Now let $\mathcal{S}=\mathcal{S}(T)$ be a full subquiver of $\Gamma(\mod \Lambda)$ with $\mathcal{S}(T)_{0}=\{\tau^{-r_{x}}P(x)\mid x\in Q_{0}.\}$
It is sufficient to show that $\mathcal{S}$ is a slice. Since $T$ is a tilting module, we have that $\mathcal{S}$ is sincere. 
Let $W$ be a path from $\tau^{-r_{x}}P(x)$ to $\tau^{-r_{y}}P(y)$ on $\Gamma(\mod \Lambda)$ and let $Z$ be an indecomposable module on $W$.
Note that $Z$ is pre-projective. Thus we may assume that $Z=\tau^{-s}P(z)$. If $s<r_{z}$, then 
By (1), we have $r_{z}>s\geq r_{x}+l_{Q}(x,z) $. This is a contradiction. If $s>r_{z}$, then we have that $r_{y}\geq s+l_{Q}(z,y)>r_{z}+l_{Q}(z,y)$. We reach a contradiction. Therefore we have $Z=\tau^{-r_{z}}P(Z)\in \mathcal{S}_{0}$. 
Consider an arrow $X\to \tau^{-r_{y}}P(y)$ in  $\Gamma(\mod \Lambda)$. We show that either $X$ or $\tau^{-}X$ is in $\mathcal{S}_{0}$.  
We may assume that $X=\tau^{-r}P(x)$. Note that one of the following holds:
\begin{itemize}
\item[(i)] There is a path from $\tau^{-r_{x}}P(x)$ to $\tau^{-r_{y}}P(y)$.
\item[(ii)] There is a path from $\tau^{-r_{y}}P(y)$ to $\tau^{-r_{x}}P(x)$.
\end{itemize} 
In the case of (i), we have $r\geq r_{x}$ and  $r_{y}\geq r+l_{Q}(x,y)\geq r_{x}+l_{Q}(x,y)$. We conclude that $r_{y}=r+l_{Q}(x,y)=r_{x}+l_{Q}(x,y)$. Therefore we have $X\in \mathcal{S}_{0}$.
Similarly, in the case of (ii), we can check that $\tau^{-}X\in \mathcal{S}_{0}$. Now same argument implies that if there is an arrow $\tau^{-r_{x}}P(x)\to Y$, then either $Y$ or $\tau Y$ is in $\mathcal{S}_{0}$. 
Finally it is obvious that if $X\in \mathcal{S}_{0}$ then $\tau^{-}X\not\in\mathcal{S}_{0}$. We conclude that $\mathcal{S}$ is a slice.  
\end{proof}
\begin{remark}
\label{mutforslice}
Let $\mathcal{S}$ be a slice in pre-projective component of $\Gamma(\mod \Lambda)$ and let $X$ be a source of $\mathcal{S}$. If $\mathcal{S}'$ is a full sub quiver
with $\mathcal{S}'_{0}=\mathcal{S}_{0}\sqcup\{\tau^{-}X\}\setminus\{X\}$, then we can check that $\mathcal{S}'$
satisfies (S2), (S3) and (S4). By the condition (S4) and the fact that $X$ is source of $\mathcal{S}$, we have that
for any irreducible map $X\to Y$, $Y$ is in $\mathcal{S}_{0}$. Hence $X$ is a submodule of finite direct sums of copies of $T_{\mathcal{S}}/X$. We conclude that
 $T_{\mathcal{S}}/X$ is sincere. In particular, $\mathcal{S}'$  satisfies (S1) and so
$\mathcal{S}'$ is a slice. In this case, we call $\mathcal{S}'$ a mutation of $\mathcal{S}$ at a source $X$ and denote
$\mathcal{S}\stackrel{X}{\to} \mathcal{S}'$.    
\end{remark}
 \begin{lemma}
 \label{sliceisfull}
 Let $T=\bigoplus_{x=1}^{n}\tau^{-r_{x}}P(x),T'=\bigoplus_{x=1}^{n}\tau^{-s_{x}}P(x)$ are two slice tilting modules with $T \geq T'$.
Then we have the following.
\begin{itemize}
\item[(1)] For any $x$, we have $r_{x}\leq s_{x}$.
\item[(2)]  If $T>T'$, then there is a path 
\[T=T_{0}\to T_{1}\to \cdots \to T_{l}=T'\]
on $\tilt \Lambda$ such that $\add\bigoplus_{k=0}^{l} T_{k}=\add \bigoplus_{x=1}^{n}\bigoplus_{r=r_{x}}^{r'_{x}}\tau^{-r}P(x).$
\end{itemize}  
 \end{lemma}
 \begin{proof}
 (1)\;Suppose that $r_{x}\geq s_{x}$. 
If $\Ext_{\Lambda}^{1}(\tau^{-s_{x}}P(x),T)\neq 0$, then by Auslander-Reiten duality, we have $\Hom_{\Lambda}(\tau^{-r_{y}}P(y),\tau^{-s_{x}+1}P(x))\neq 0$ for some $y$. Thus there is a chain of irreducible maps
from $\tau^{-r_{y}}P(y)$ to $\tau^{-s_{x}+1}P(x)$. In particular, there exists a chain of irreducible maps from $\tau^{-r_{y}}P(y)$ to $\tau^{-r_{x}}P(x)$.
 By Lemma\;\ref{chain-condition}, we obtain
\[ r_{y}+l_{Q}(y,x)= r_{x}>s_{x}-1\geq r_{y}+l_{Q}(y,x).\]
This is a contradiction. Therefore $\Ext_{\Lambda}^{1}(\tau^{-s_{x}}P(x),T)=0$. Since $\Ext_{\Lambda}^{1}(T,T')=0$, we have that $s_{x}=r_{x}$.

(2)\;Assume that $T=T_{\mathcal{S}}>T'=T_{\mathcal{S}'}$. We prove the assertion by using induction on $t=t(T,T'):=\Sigma _{x=1}^{n}(s_{x}-r_{x})$.
If $t=1$, then the assertion is follows from $(1)$.  We assume that the assertion holds for any $t<m$ $(m\geq 2)$, and consider the case $t=m$. 
Let $\psi =\{a\mid \tau^{-r_{a}}P(a)\ \mathrm{is\ a\ source\ of}\ \mathcal{S}\}$. 
Now we suppose that $r_{a}=s_{a}$ for any $a\in \psi$. Since for any $x\in \{1,2,\dots,n\}\setminus \psi$, there is a element
$a$ in $\psi$ such that there is a chain of irreducible maps from $\tau^{-r_{a}}P(a)$ to $\tau^{-r_{x}}P(x)$. By using  
 Lemma\;\ref{chain-condition}\;(2), we have that \[ s_{x}\geq r_{x}=r_{a}+l_{Q}(a,x)=s_{a}+l_{Q}(a,x).\]
Therefore, by Lemma\;\ref{chain-condition}\;(3), we obtain $s_{x}=r_{x}$. Thus $T=T'$. This is a contradiction.
 Hence there is a element $a\in \psi$ such that $s_{a}>r_{a}$. Then we have $T_{\mathcal{S}}\rightarrow T_{\mathcal{S}''}>T_{\mathcal{S}'}$,
where $\mathcal{S}''$ is a complete slice obtained from $\mathcal{S}$ by applying mutation at $\tau^{-r_{a}}P(a)$.
Thus we have $T_{\mathcal{S}''}=\bigoplus_{x=1}^{n}\tau^{-r''_{x}}P(x)$ where $r''_{x}=r_{x}$ for any $x\in\{1,\dots\, n\}\setminus\{a\}$
and $r''_{a}=r_{a}+1$.
Since $t(T'',T')=m-1$, we have  that there is a path 
\[T''=T_{1}\to T_{2}\to \cdots \to T_{l}=T'\]
on $\tilt \Lambda$ such that $\add\bigoplus_{k=1}^{l} T_{k}=\add \bigoplus_{x=1}^{n}\bigoplus_{r=r''_{x}}^{r'_{x}}\tau^{-r}P(x).$
Therefore the path
\[T=T_{0}\to T_{1}\to\cdots\to T_{l}=T'\]
satisfies the desired property. 
 \end{proof}
\section{Lengths of maximal green sequences for quivers of type $A_{n}$}
In this section we always assume that $Q$ is a quiver of type $A_{n}$ and show the following.
\begin{theorem}
\label{interval conj for type a} We have
$ \{l\in \Z_{\geq 0} \mid \mathrm{green}_{l}(Q)\neq \emptyset\}=[n,\frac{n(n+1)}{2}].$
\end{theorem}
Let $\Gamma=\Gamma(\mod \Lambda)$ be the Auslander-Reiten quiver of $\mod \Lambda$. Then 
\[\Gamma_{1}=\Gamma_{1}(1)\sqcup\Gamma_{1}(2)\]
where $\Gamma_{1}(1)$ (resp. $\Gamma_{1}(2)$) is the set of all arrows of following form:
\[\tau^{-r_{i}}P(i)\to \tau^{-r_{j}}P(j)\ \mathrm{with\ }j=i-1\ (\mathrm{resp.}\ \tau^{-r_{i}}P(i)\to \tau^{-r_{j}}P(j)\ \mathrm{with\ }j=i+1).\]
 We denote by $X$ a unique sincere indecomposable $\Lambda$-module. Since we have
 \[\Hom_{\Lambda}(P(1),X)\neq 0\neq \Hom_{\Lambda}(P(n),X)\]
 and $\Gamma$ is standard (see \cite[Section\;2.4]{R}), we have that there is a unique path
 \[P(1)=\tau^{-r_{1}}P(1)\to \tau^{-r_{2}}P(2)\to\cdots \to\tau^{-r_{i}}P(i)=X\] in $\Gamma$ and there is a unique path
 \[P(n)=\tau^{-r_{n}}P(n)\to \tau^{-r_{n-1}P(n-1)}\to\cdots \to\tau^{-r_{i}}P(i)=X\] in $\Gamma$.
 We set $\Xi :=\{\tau^{-r}P(j)\in \ind \mod \Lambda \mid r<r_{j}\}$ and $\bar{\Xi}:=\{\tau^{-r}P(j)\in \ind \mod\Lambda\mid r \leq r_{j}\} $.
 \begin{lemma}
 \label{form of max for x}
 $\Ext_{\Lambda}^{1}(X,\Xi)\neq 0$. Moreover, $\bigoplus_{j=1}^{n}\tau^{-r_{j}}P(j)$ is a maximum element of
 $\stilt_{X}(\Lambda)$.
 \end{lemma}
 \begin{proof}
 By following rule, we regard $\bar{\Xi}$ as a full subquiver of $\Z\times \Z$.
 \begin{itemize}
 \item $X=(0,0)$.
 \item If $Y\to Z \in \Gamma_{1}(1)$, then $Y=Z-(1,0)$.
 \item If $Y\to Z \in \Gamma_{1}(2)$, then $Y=Z-(0,1)$.
 \end{itemize}
 By definition, we have $\bar{\Xi}\subset \{(-p,-q)\mid 0\leq p\leq n-i,\ 0\leq q \leq i-1  \}$ and
 $\Xi\subset\{(-p,-q)\mid 0<p\leq n-i,\ 0< q\leq i-1\}$. 
Let $X'=(-p,-q)\in \bar{\Xi}$. We have that
\[(-p,-q)\to (-p+1,-q)\to\cdots\to (0,-q)\not\equiv 0\ \mathrm{mod}\stackrel{\mathrm{mesh}}{\sim},\]
where $\stackrel{\mathrm{mesh}}{\sim }$ is the mesh relation.
This means that $\Hom_{\Lambda}(X',\tau^{-r_{i-q}}P(i-q))\neq 0$. We remark that $X$ is faithful and 
$\Hom_{\Lambda}(\tau^{-r_{i-q}}P(i-q),X)\neq 0$. Therefore Lemma\;\ref{hv} implies that $0\neq f\in \Hom_{\Lambda}(\tau^{-r_{i-q}}P(i-q), X)$ is  injective. In particular, we have that \[\Hom_{\Lambda}(X',X)=0.\]
Note that for any $X''\in \Xi$, we have that $\tau^{-1}X''\in \bar{\Xi}$.  Therefore we have
 \[\dim\Ext_{\Lambda}^{1}(X,X'')=\dim\Hom_{\Lambda}(\tau^{-1}X'',X)\neq 0\]
 for any $X''\in \Xi$.  Since there is no path from $X$ to $\tau^{-r_{j}+1}P(j)$, we have
 \[\dim\Ext_{\Lambda}^{1}(\tau^{-r_{j}}P(j),X)=\dim\Hom_{\Lambda}(X,\tau^{-r_{j}+1}P(j))=0.\]
 Since there is no path from $\tau^{-r_{j}-1}P(j)$ to $X$, we also have
 \[\dim\Ext_{\Lambda}^{1}(X,\tau^{-r_{j}}P(j))=\dim\Hom_{\Lambda}(\tau^{-r_{j}-1}P(j),X)=0.\]
  We conclude that $\bigoplus_{j=1}^{n}\tau^{-r_{j}}P(j)$ is maximum element of $\stilt_{X}(\Lambda)$.  
 \end{proof}
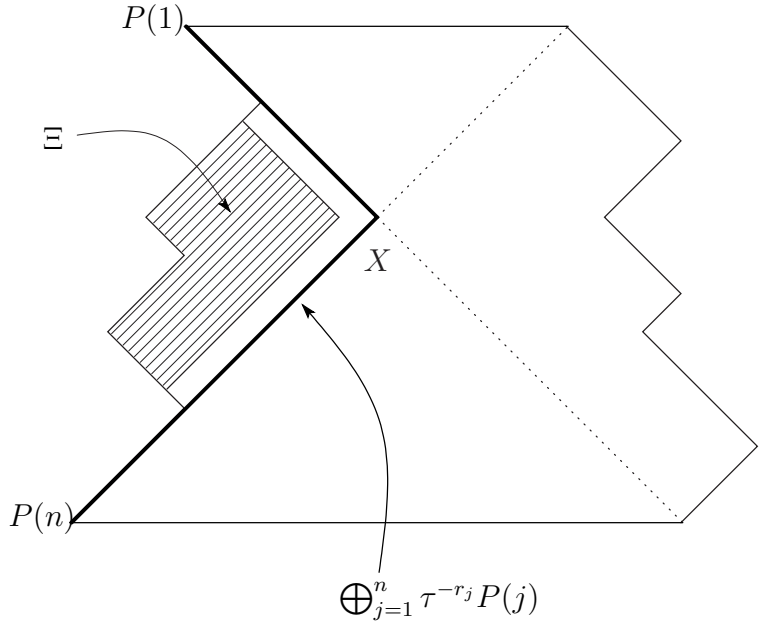
\[ 
{\unitlength 0.1in%
\begin{picture}( 39.2000, 30.6000)( 14.7000,-39.8000)%
%
\special{pn 8}%
\special{pa 1790 3600}%
\special{pa 2390 3000}%
\special{pa 1990 2600}%
\special{pa 2390 2200}%
\special{pa 2190 2000}%
\special{pa 2790 1400}%
\special{pa 2390 1000}%
\special{fp}%
%
\special{pn 8}%
\special{pa 4390 1000}%
\special{pa 4990 1600}%
\special{pa 4590 2000}%
\special{pa 4990 2400}%
\special{pa 4790 2600}%
\special{pa 5390 3200}%
\special{pa 4990 3600}%
\special{fp}%
%
\special{pn 8}%
\special{pa 3400 2000}%
\special{pa 4400 1000}%
\special{dt 0.045}%
%
\special{pn 8}%
\special{pa 3400 2000}%
\special{pa 5000 3600}%
\special{dt 0.045}%
\put(33.3000,-22.8000){\makebox(0,0)[lb]{$X$}}%
%
\special{pn 8}%
\special{pa 2700 1500}%
\special{pa 3200 2000}%
\special{fp}%
%
\special{pn 8}%
\special{pa 3200 2000}%
\special{pa 2300 2900}%
\special{fp}%
%
\special{pn 4}%
\special{pa 3000 1800}%
\special{pa 2100 2700}%
\special{fp}%
\special{pa 3030 1830}%
\special{pa 2130 2730}%
\special{fp}%
\special{pa 3060 1860}%
\special{pa 2160 2760}%
\special{fp}%
\special{pa 3090 1890}%
\special{pa 2190 2790}%
\special{fp}%
\special{pa 3120 1920}%
\special{pa 2220 2820}%
\special{fp}%
\special{pa 3150 1950}%
\special{pa 2250 2850}%
\special{fp}%
\special{pa 3180 1980}%
\special{pa 2280 2880}%
\special{fp}%
\special{pa 2970 1770}%
\special{pa 2070 2670}%
\special{fp}%
\special{pa 2940 1740}%
\special{pa 2040 2640}%
\special{fp}%
\special{pa 2910 1710}%
\special{pa 2010 2610}%
\special{fp}%
\special{pa 2880 1680}%
\special{pa 2380 2180}%
\special{fp}%
\special{pa 2850 1650}%
\special{pa 2350 2150}%
\special{fp}%
\special{pa 2820 1620}%
\special{pa 2320 2120}%
\special{fp}%
\special{pa 2790 1590}%
\special{pa 2290 2090}%
\special{fp}%
\special{pa 2760 1560}%
\special{pa 2260 2060}%
\special{fp}%
\special{pa 2730 1530}%
\special{pa 2230 2030}%
\special{fp}%
%
\special{pn 4}%
\special{pa 1830 1570}%
\special{pa 1864 1565}%
\special{pa 1898 1561}%
\special{pa 1931 1557}%
\special{pa 1965 1553}%
\special{pa 1998 1550}%
\special{pa 2031 1548}%
\special{pa 2063 1547}%
\special{pa 2095 1548}%
\special{pa 2126 1550}%
\special{pa 2157 1553}%
\special{pa 2187 1558}%
\special{pa 2216 1566}%
\special{pa 2244 1575}%
\special{pa 2271 1587}%
\special{pa 2298 1600}%
\special{pa 2323 1616}%
\special{pa 2348 1633}%
\special{pa 2372 1652}%
\special{pa 2395 1673}%
\special{pa 2418 1695}%
\special{pa 2441 1718}%
\special{pa 2463 1742}%
\special{pa 2505 1794}%
\special{pa 2526 1821}%
\special{pa 2546 1849}%
\special{pa 2567 1878}%
\special{pa 2607 1936}%
\special{pa 2626 1965}%
\special{pa 2630 1970}%
\special{fp}%
%
\special{pn 4}%
\special{pa 2626 1965}%
\special{pa 2630 1970}%
\special{fp}%
\special{sh 1}%
\special{pa 2630 1970}%
\special{pa 2604 1905}%
\special{pa 2597 1928}%
\special{pa 2573 1930}%
\special{pa 2630 1970}%
\special{fp}%
\put(16.5000,-16.4000){\makebox(0,0)[lb]{$\Xi$}}%
%
\special{pn 8}%
\special{pa 2400 1000}%
\special{pa 4400 1000}%
\special{fp}%
%
\special{pn 8}%
\special{pa 1800 3600}%
\special{pa 5000 3600}%
\special{fp}%
\put(14.7000,-36.6000){\makebox(0,0)[lb]{$P(n)$}}%
\put(20.6000,-10.5000){\makebox(0,0)[lb]{$P(1)$}}%
%
\special{pn 20}%
\special{pa 1800 3600}%
\special{pa 3400 2000}%
\special{pa 2400 1000}%
\special{fp}%
%
\special{pn 8}%
\special{pa 3410 3860}%
\special{pa 3415 3827}%
\special{pa 3419 3793}%
\special{pa 3424 3760}%
\special{pa 3428 3726}%
\special{pa 3444 3594}%
\special{pa 3447 3561}%
\special{pa 3451 3495}%
\special{pa 3453 3463}%
\special{pa 3454 3431}%
\special{pa 3455 3398}%
\special{pa 3455 3366}%
\special{pa 3454 3335}%
\special{pa 3453 3303}%
\special{pa 3444 3210}%
\special{pa 3439 3179}%
\special{pa 3433 3149}%
\special{pa 3419 3089}%
\special{pa 3410 3060}%
\special{pa 3400 3031}%
\special{pa 3389 3002}%
\special{pa 3377 2974}%
\special{pa 3364 2946}%
\special{pa 3350 2918}%
\special{pa 3336 2891}%
\special{pa 3320 2863}%
\special{pa 3304 2836}%
\special{pa 3287 2809}%
\special{pa 3269 2783}%
\special{pa 3251 2756}%
\special{pa 3232 2730}%
\special{pa 3172 2652}%
\special{pa 3151 2626}%
\special{pa 3130 2601}%
\special{pa 3109 2575}%
\special{pa 3087 2550}%
\special{pa 3066 2524}%
\special{pa 3022 2474}%
\special{pa 3010 2460}%
\special{fp}%
%
\special{pn 8}%
\special{pa 3022 2474}%
\special{pa 3010 2460}%
\special{fp}%
\special{sh 1}%
\special{pa 3010 2460}%
\special{pa 3038 2524}%
\special{pa 3045 2500}%
\special{pa 3069 2498}%
\special{pa 3010 2460}%
\special{fp}%
\put(32.0000,-41.1000){\makebox(0,0)[lb]{$\bigoplus_{j=1}^{n}\tau^{-r_{j}}P(j)$}}%
\end{picture}}\]\vspace{5pt}
 
 Let $X=\tau^{s_{n-i}}I(n-i)\to \tau^{s_{n-i+1}}I(n-i+1)\to \cdots\to \tau^{s_{n}}I(n)=I(n) $ be a unique path
 from $X$ to $I(n)$ and $X=\tau^{s_{n-i}}I(n-i)\to \tau^{s_{n-i-1}}I(n-i-1)\to \cdots\to \tau^{s_{1}}I(1)=I(1)$
 a unique path from $X$ to $I(1)$. Then we have following dual result for Lemma\;\ref{form of max for x}.
\begin{lemma}
\label{form of min for x}
$\bigoplus_{j=1}^{n}\tau^s_{j}I(j)$ is a minimum element of $\stilt_{X}(\Lambda)$.
\end{lemma}
 
We now prove Theorem\;\ref{interval conj for type a}
\begin{proof}
Let $X$ be a unique indecomposable sincere module, $T$ a maximum element of $\stilt_{X}(\Lambda)$ and $T'$ a minimum element of $\stilt_{X}(\Lambda)$.
Then there is a vertex $i$ of $Q$ and a non negative integer $r$ such that $X\simeq \tau^{-r}P(i)$. Then by Theorem\;\ref{stilting reduction}, we have that
 \[(\ast)\ \ \stilt_{X}(\Lambda)\simeq \stilt(\Lambda_{1}\times \Lambda_{2})=\stilt(\Lambda_{1})\times \stilt(\Lambda_{2}),\]
where $\Lambda_{1}=k(1\leftarrow 2\leftarrow\cdots \leftarrow i-1)$ and $\Lambda_{2}=k(i+1\rightarrow i+2\rightarrow\cdots \rightarrow n)$.

We first show that the assertion holds in the case $n\leq 5$.
It is easy to check in the case $n=1$. If $n=2,3$, then the assertion directly follows from Example\;\ref{example1} and Example\;\ref{exstilta3}.
We assume that $n=4$. Since $\stilt(\Lambda^{\mathrm{op}})\simeq \stilt(\Lambda)^{\mathrm{op}}$, we may assume that $1$ is source of $Q$.
Then $\Lambda/P(1)\in \stilt (\Lambda)$ and \[\{T\in \stilt(\Lambda)\mid T\leq \Lambda/P(1)\}=\{T\in \stilt(\Lambda)\mid 1\not\in \mathrm{supp}(T)\}\simeq
 \stilt(k(Q\setminus\{1\}))\cdots (\ast\ast).\]
 Therefore, for any $l'\in [3,6]$, there is a path from $\Lambda$ to $0$ in $\stilt(\Lambda)$ with length $l'+1$. Therefore for any $l\in[4,7]$, there is a maximal green sequence for $Q$ with length $l$.
 Lemma\;\ref{sliceisfull} and $(\ast)$ also implies that for any $l\in[9,10]$, there exists a maximal green sequence for $Q$ with length $l$. In fact,
we have that  
\[\frac{(i-1)i+(n-i)(n-i+1)}{2}-(n-1)\geq 1\ (n=4).\]
Hence, it is sufficient to show that there is a maximal green sequence for $Q$ with length $8$.
 Then we have that $\bigoplus_{i\in \mathrm{supp}(P(1))}I(i)$ is a minimum element of $\stilt_{P(1)}(\Lambda)$.
Since $\stilt_{P(1)}(\Lambda)\simeq\stilt(k(Q\setminus\{1\}))$, for any $l'\in[3,6]$, there is a path from $\Lambda$ to $\bigoplus_{i\in \mathrm{supp}(P(1))}I(i)$
with length $l'$.
We conclude that for any $l\in [3+\#\mathrm{supp}(P(1)),6+\#\mathrm{supp}(P(1))]$, there exists a maximal green sequence for $Q$ with length $l$.
Then the assertion follows from the fact that $2\leq \#\mathrm{supp}(P(1))\leq 4$.
Similarly, we can check that the assertion holds in the case $n=5$.  
We now prove the assertion for $n>5$ by using induction on $n$. 
Therefore we may assume that $n \geq 6$ and the assertion holds for any $k<n$.

Note that $(\ast\ast)$ and hypothesis of induction imply that for any $l\in [n,\frac{n(n-1)}{2}+1]$, there is a path from $\Lambda$ to $0$ in $\stilt(\Lambda)$ with length $l'+1$.
By hypothesis of induction, we see that for any $l'\in [n-1,\frac{(i-1)i}{2}+\frac{(n-i)(n-i+1)}{2}]$, there is a path from $T$ to $T'$ with length $l'$.
Hence Lemma\;\ref{sliceisfull} shows that for any \[l\in[\frac{n(n+1)}{2}-(\frac{(i-1)i+(n-i)(n-i+1)}{2}-(n-1)),\frac{n(n+1)}{2}],\] there exists a maximal green sequence for $Q$ with length $l$.
Therefore it is sufficient to show that 
\[\frac{n(n+1)}{2}-(\frac{(i-1)i+(n-i)(n-i+1)}{2}-(n-1))\leq \frac{(n-1)n}{2}+2.\]
If $n=2k$ $(k\geq 3)$ is  even, then we have
\[\begin{array}{lll}
\frac{n(n+1)}{2}-(\frac{(i-1)i+(n-i)(n-i+1)}{2}-(n-1)) &\leq & \frac{n(n+1)}{2}-(\frac{(k-1)k+(n-k)(n-k+1)}{2}-(n-1))\\
& =& k^{2}+3k-1 \\
& \leq & 2k^{2}-k+2=\frac{(n-1)n}{2}+2.
\end{array}
\]
If $n=2k+1$ $(k\geq 3)$ is odd, then we have
\[\begin{array}{lll}
\frac{n(n+1)}{2}-(\frac{(i-1)i+(n-i)(n-i+1)}{2}-(n-1)) &\leq & \frac{n(n+1)}{2}-(\frac{k(k+1)+(n-k-1)(n-k)}{2}-(n-1))\\
& =& k^{2}+4k+1 \\
& \leq & 2k^{2}+k+2=\frac{(n-1)n}{2}+2.
\end{array}
\]
 
\end{proof}
\section{Lengths of maximal green sequences for a quiver of type $\tilde{A}_{n,1}$}
In this section, we consider a path algebra $\Lambda=kQ$ of type $\tilde{A}_{n,1}$. 
\begin{lemma}
\label{having projective}
Let $\Lambda=T_{0}\rightarrow T_{1} \rightarrow\cdots \rightarrow T_{l}=0 $ be a path in $\stilt(\Lambda)$.
Assume that $T_{r-1}\in\tilt \Lambda$ and $T_{r}\notin \tilt \Lambda$ with $\mathrm{supp}(T_{r})=Q_{0}\setminus\{i\}$.
Then $i\neq n$ and $P(i+1)\in \add T_{r}$.
\end{lemma}
\begin{proof}
We first shows that $T_{r}$ has no non-zero pre-injective direct summand. If $T_{r}$ has non-zero pre-injective direct summand,
then there is an integer $k\in\{1,2,\dots,r-1\}$ such that $T_{k-1}$ has no non-zero pre-injective direct summand and $T_{k}$
has non-zero pre-injective direct summand $I$. By Proposition\;\ref{representative}, we have that $T_{k-1}$ has at least two pre-projective
indecomposable direct summands. Therefore $T_{k}$ has non-zero pre-projective direct summand $P$. We conclude that $T_{k}$ has both a non-zero pre-projective direct summand $P$ and a pre-injective direct summand $I$.
This is a contradiction. In particular, $T_{r}$ has no non-zero pre-injective direct summand. Now Proposition\;\ref{representative} implies that there is a pair $(s,j)\in Z_{\geq 0}\times Q_{0}$ 
such that $\Tau^{s} T_{r}\in \stilt_{X_{j}}\Lambda$.
Suppose that $i=n$. Then, by definition of $r$, we have $s>0$. Thus $\tau^{-s+1}P(0)\oplus \tau^{-s+1}P(n)\in \add T_{r}$. Since $\tau^{\leq 0}P(0)$ is sincere (see Remark\;\ref{prin}), we reach a contradiction.
Therefore $i\neq n$. If $s>0$, then $\tau^{-s}P(j)\oplus \tau^{-s+1}P(j+1)\in \add T_{r}$. Since $\tau^{<0} P(j') $ is sincere for any $j'\in Q_{0}$, we reach a contradiction. Therefore we have that $s=0$ and $j=i$.   
\end{proof}
\begin{proposition}
Let $\Lambda=T_{0}\rightarrow T_{1} \rightarrow\cdots \rightarrow T_{l}=0 $ be a path in $\stilt(\Lambda)$.
Then we have \[l\leq \frac{n(n+3)}{2}=n+\frac{n(n+1)}{2}\]
\end{proposition}
\begin{proof}
We may assume that $T_{r-1}\in \tilt \Lambda$ and $T_{r}\not\in\tilt \Lambda$.
Then there is a unique vertex $i$ of $Q$ such that $\Hom_{\Lambda}(P(i),T_{r})=0$, or equivalently $(\dimvec\;T_{r})_{i}=0.$
By Lemma\;\ref{having projective}, $T_{r}$ has an indecomposable direct summand $P(i+1)$.
Note that $\add\bigoplus_{k=0}^{r-1}T_{k}\subset \Ker\Ext_{\Lambda}^{1}(-,T_{r-1})$ and $\add\bigoplus_{k=r}^{l}T_{k}\subset \mathrm{Fac}\;T_{r}$.
Therefore we have \[l=\#\ind \add \bigoplus_{k=0}^{l}T_{k}\leq\#\ind\Ker\Ext_{\Lambda}^{1}(-,T_{r-1})\cup \mathrm{Fac}\;T_{r}.\]
If $\Ext_{\Lambda}^{1}(X,T_{r-1})=0$ and $X\in \mathrm{Fac}\;T_{r}$, then $X\in \add T$. In particular, we have
\[\#\ind\Ker\Ext_{\Lambda}^{1}(-,T_{r-1})\cup \mathrm{Fac}\;T_{r}=\#\ind\Ker\Ext_{\Lambda}^{1}(-,T_{r})+ \#\ind\mathrm{Fac}\;T_{r}-n.\]
By Auslander-Reiten duality, we have
\[\Ext_{\Lambda}^{1}(X,T_{r-1})=0\Leftrightarrow \Hom_{\Lambda}(T_{r-1},\tau X)=0.\]
Hence we have
\[\#\ind\Ker\Ext_{\Lambda}^{1}(-,T_{r-1})-n-1\leq \#\ind\Ker \Hom_{\Lambda}(T_{r-1},-).\]
We conclude that
\[l\leq \#\ind\Ker \Hom_{\Lambda}(T_{r-1},-)+\#\ind\mathrm{Fac}\;T_{r}+1.\]
Now we let $\chi:=\{X\in \Ker \Hom_{\Lambda}(T_{r-1},-)\mid \Hom_{\Lambda}(P(i),X)\neq 0 \}.$
We claim that $\#\ind \chi\leq n-1$. Let $\bar{\chi}:=\{X\in \mod \Lambda\mid \Hom_{\Lambda}(P(i+1),X)=0, \Hom_{\Lambda}(P(i),X)\neq 0 \}$.
We note that $I(i)\in \bar{\chi}$ and $\#\ind \chi\leq n$.
Since $P(i+1)\in\add T_{r}$, we have $P(i+1)\in \add T_{r-1}$. 
   In particular, $\chi\subset\bar{\chi}$.
Therefore it is sufficient to show that $ \chi\neq \bar{\chi}$.
Since $T_{r-1}\in \tilt(\Lambda)$, there is an indecomposable pre-projective direct summand $P=\tau^{-r}P(j)\not\simeq P(i+1) $ of $T_{r-1}$.
If $r>0$, then $\tau^{r}I(i)$ is sincere and we have 
\[\Hom_{\Lambda}(P,I(i))\simeq\Hom_{\Lambda}(P(j),\tau^{r}I(i))\neq 0.\]
 Thus $I(i)\not\in \chi$. Therefore we may assume that $P=P(j)$. Let $X$ be a unique sincere indecomposable $k(Q\setminus\{i+1\})$-module.
 Then we can regard $X$ as an indecomposable $\Lambda$-module. Since $j\neq i+1$, we have 
 \[\Hom_{\Lambda}(P(j),X)\neq 0.\] Therefore we have $X\not\in \chi$. We conclude that $\chi\neq \bar{\chi}$.
 
 Now we have 
\[\begin{array}{lll}
l&\leq &\#\ind \chi+\#\ind(\Ker \Hom_{\Lambda}(T_{r-1},-)\cup\mathrm{Fac}\;T_{r})\cap \Ker\Hom_{\Lambda}(P(i),-)+1\\
&\leq & n-1+\frac{n(n+1)}{2}+1=\frac{n(n+3)}{2}.
\end{array}\] 
\end{proof}
\begin{theorem} Let $Q=Q^{(n)}$ be a quiver of type $\tilde{A}_{n,1}$. Then we have
\[\{l\in \Z_{\geq 0}\mid \mathbf{green}_{l}(Q)\neq \emptyset\}=[n+1,\frac{n(n+3)}{2}].\]
\end{theorem}
\begin{proof}
Let $1\leq i\leq n$. We consider $\stilt_{P(i)}(\Lambda)\simeq \stilt(kQ\setminus\{i\})$. Note that $\stilt_{P(i)}(\Lambda)$ has a maximum element $\Lambda$ and a minimum element $I_{i}$, 
where $I_{i}$ is a injective tilting module of $k(i\rightarrow \cdots\rightarrow n)$.\;($I_{i}$ is not injective as a $\Lambda$-module)
Then we can check that there is a unique path  from $I_{i}$ to  $0$ and its length is equal to $n-i+1$. Since $\stilt_{P(i)}(\Lambda)\simeq \stilt(kQ\setminus\{i\})$
, for any maximal green sequence $\mathbf{i}$ of $Q\setminus\{i\}$, there is a path from $\Lambda$ to $I_{i}$ whose length is equal to the length of $\mathbf{i}$. Therefore Theorem\;\ref{interval conj for type a}
implies that for any $l'\in [n,\dots,\frac{n(n+1)}{2}]$,
there exists a path from $\Lambda$ to $I_{i}$ in $\stilt(\Lambda)$ with length $l'$. Hence, for any $l\in\{n+n-i+1,\dots,\frac{n(n+1)}{2}+n-i+1\}$, there exists a maximal green sequence
with length $l$. 
\end{proof}
\begin{remark} Under Conjecture\;\ref{interval conj}, T.~Br$\ddot{\mathrm{u}}$stle, G.~Dupont and M.~P$\acute{\mathrm{e}}$rotin
already calculated $\mathrm{max}\{l\in Z_{\geq 0}\mid \mathbf{green}_{l}(Q^{(n)})\neq \emptyset \}$
for $n<8$. 
\end{remark}

\end{document}